\documentclass[letterpaper]{amsart}

\usepackage[utf8]{inputenc}
\usepackage{amsmath}
\usepackage{amssymb}
\usepackage{amsthm}
\usepackage{tikz-cd}
\usepackage{mathrsfs}
\usepackage{enumitem}
\usepackage[alphabetic]{amsrefs}
\usepackage{tikz}
\usepackage{float}
\usepackage{hyperref}

\theoremstyle{plain}
\newtheorem{teo}{Theorem}[section]
\newtheorem{prop}[teo]{Proposition}
\newtheorem{lemma}[teo]{Lemma}
\newtheorem{cor}[teo]{Corollary}
\theoremstyle{definition}

\newtheorem{defin}[teo]{Definition}

\newtheorem{example}[teo]{Example}
\newtheorem{question}[teo]{Question}
\theoremstyle{remark}
\newtheorem{remark}[teo]{Remark}
\newtheorem{claim}[teo]{Claim}
\numberwithin{equation}{section}

\newcommand{\A}{\mathcal{A}}

\newcommand{\C}{\mathbb{C}}
\newcommand{\Cchamb}{\mathcal{C}}

\newcommand{\D}{\mathcal{D}}
\newcommand{\DC}{\mathrm{D}}
\newcommand{\E}{\mathscr{E}}
\newcommand{\Ftors}{\mathcal{F}}
\newcommand{\HC}{\mathcal{H}}
\newcommand{\K}{\mathcal{K}}
\renewcommand{\L}{\mathscr{L}}
\newcommand{\M}{\mathfrak{M}}
\newcommand{\m}{\mathfrak{m}}

\renewcommand{\O}{\mathscr{O}}
\renewcommand{\P}{\mathbb{P}}
\newcommand{\Pslicing}{\mathcal{P}}
\newcommand{\Q}{\mathbb{Q}}
\newcommand{\R}{\mathbb{R}}
\newcommand{\Ttors}{\mathcal{T}}
\newcommand{\U}{\mathcal{U}}
\newcommand{\Ucover}{\mathfrak{U}}
\newcommand{\Z}{\mathbb{Z}}

\let\Re\relax
\let\Im\relax
\DeclareMathOperator{\Amp}{Amp}
\DeclareMathOperator{\Aut}{Aut}
\DeclareMathOperator{\ch}{ch}
\DeclareMathOperator{\Coh}{Coh}
\DeclareMathOperator{\Exc}{Exc}
\DeclareMathOperator{\Ext}{Ext}
\DeclareMathOperator{\Hom}{Hom}
\DeclareMathOperator{\Im}{Im}
\DeclareMathOperator{\Nef}{Nef}
\DeclareMathOperator{\NS}{NS}
\DeclareMathOperator{\Re}{Re}
\DeclareMathOperator{\rk}{rk}

\DeclareMathOperator{\Spec}{Spec}
\DeclareMathOperator{\Stab}{Stab}
\DeclareMathOperator{\Supp}{Supp}

\DeclareMathOperator{\Tot}{Tot}

\newcommand{\Art}{\mathrm{Art}}
\newcommand{\Bl}{\mathrm{Bl}}
\newcommand{\Def}{\mathrm{Def}}
\newcommand{\EXT}{\mathcal{E}\kern -.5pt xt}
\newcommand{\GL}{\mathrm{GL}}
\newcommand{\HOM}{\mathcal{H}\kern -.5pt om}
\newcommand{\id}{\mathrm{id}}

\newcommand{\MC}{\mathrm{MC}}
\newcommand{\Set}{\mathrm{Set}}

\newcommand{\abs}[1]{\left\lvert#1\right\rvert}

\newcommand{\smallddots}{\tikz{\fill (0,0) circle [radius=0.02] (0.12,-0.08) circle [radius=0.02] (-0.12,0.08) circle [radius=0.02];}}

\begin{document}

\title{Some examples of moduli spaces on surfaces via wall-crossing}
\author{Nicolás Vilches}
\address{Department of Mathematics, Columbia University, 2990 Broadway, New York, NY 10027, USA}
\email{nivilches@math.columbia.edu}
\begin{abstract}
We describe new explicit examples of moduli spaces of Bridgeland semi\-stable objects on surfaces, parametrizing objects whose numerical class agrees with the class of a point. This follows ideas of Tramel and Xia, using stability conditions constructed in our previous work.

Our main technical tools are a careful analysis of the wall-crossing from the geometric chamber, and explicit models for the differential graded Lie algebra governing the local structure of the moduli spaces.
\end{abstract}

\maketitle

\tableofcontents

\section{Introduction}

Bridgeland stability conditions are a key tool for the study of derived categories of varieties. Given a smooth, projective variety $X$, we get a \emph{stability manifold} $\Stab(X)$ consisting on the (numerical) stability conditions on $X$. Conjecturally, $\Stab(X)$ is non-empty for any variety $X$.

If $\sigma=(Z, \A)$ is a stability condition on $\DC^b(X)$, then $\sigma$ defines a slope function on $\A$, and so a notion of semistability for objects in $\A$. Under some technical assumptions, we get moduli spaces $M_\sigma(v)$ parametrizing (S-equivalence classes of) $\sigma$-semistable objects with numerical vector $v$. 

A natural question is trying to \emph{describe} the moduli spaces $M_\sigma(v)$. For example, determining whether they are non-empty, the number of irreducible components, their singularities, and so on. It is not surprising that such a general question has been tackled extensively; for specific varieties, some examples include \citelist{\cite{AB13} \cite{BM14} \cite{Tod13} \cite{Xia18} \cite{TX22} \cite{Cho24} \cite{AS25}}. 

There is a general strategy that one can use to understand the moduli spaces $M_\sigma(v)$, cf. \cite{BM23}*{p. 2173}. We first identify an auxiliary stability condition $\tilde{\sigma}$ lying in the same connected component of $\Stab(X)$, for which $M_{\tilde{\sigma}}(v)$ is well understood. Standard examples include the \emph{large volume limit}, where $M_{\tilde{\sigma}}(v)$ agrees with moduli space of Gieseker semistable sheaves; or the \emph{geometric chamber}, where $M_{\tilde{\sigma}}([pt])$ is isomorphic to the original variety $X$. 

From here, we need a way to relate the moduli space $M_{\tilde{\sigma}}(v)$ with $M_\sigma(v)$. To do so, we look at the \emph{walls} of $\Stab(X)$ with respect to $v$: a locally finite set of real codimension 1 submanifolds dividing $\Stab(X)$ into \emph{chambers}. On each chamber, the semistable objects remain the same. This way, we pick a path $\sigma_t$ from $\tilde{\sigma}$ to $\sigma$ in $\Stab(X)$, and we study how the moduli space $M_{\sigma_t}(v)$ changes at each wall. 

Our goal is to use this approach to describe new explicit examples of moduli spaces of Bridgeland semistable objects on surfaces. Our first set of examples is induced by the contraction of disjoint rational curves on a surface.

\begin{teo} \label{teo:intro_disjoint}
Let $S$ be a smooth, projective surface. Assume that there are disjoint curves $C_1, \dots, C_r \subseteq S$, such that each $C_i$ is a smooth, rational curve of self-intersection $-n_i$, for some $n_i \geq 3$. Then, there exists a stability condition $\sigma \in \Stab(S)$, contained in the closure of the geometric chamber, satisfying the following:
\begin{enumerate}
\item There are $r$ walls $W_1, \dots, W_r$ with respect to $[pt]$ passing through $\sigma$, with transversal intersection.
\item The good moduli space $M_\sigma([pt])$ is isomorphic to the surface $T$ obtained from $S$ by contracting each of the $r$ curves $C_1, \dots, C_r$ to a cyclic quotient singularity $\frac{1}{n_i}(1, 1)$.
\item For each of the $2^r$ chambers determined by the walls $W_i$, the moduli space of semistable objects with numerical class $[pt]$ is isomorphic to $S \cup \P^{n_{i_1}-1} \cup \dots \cup \P^{n_{i_s}-1}$ for some $1 \leq i_1 <\dots <i_s \leq r$. The surface $S$ is glued along each $\P^{n_i-1}$ by identifying $C_i \subset S$ with a rational normal curve, and there are no further identifications.
\end{enumerate}
\end{teo}

The second set of examples arises from the contraction of a chain of two smooth, rational curves to a cyclic quotient singularity.

\begin{teo} \label{teo:intro_inter}
Let $S$ be a smooth, projective surface. Assume that there are two smooth, rational curves $C_1, C_2 \subseteq S$ intersecting transversally at a single point, with $C_i^2=-n_i$ and $n_i \geq 3$. Then, there exists a stability condition $\sigma \in \Stab(S)$, contained in the closure of the geometric chamber, satisfying the following:
\begin{enumerate}
\item There are three walls $W_1, W_2, W_{12}$ with respect to $[pt]$ passing through $\sigma$, dividing $\Stab(S)$ around $\sigma$ into six regions.

\item For one of the chambers, the moduli space of semistable objects with numerical class $[pt]$ is isomorphic to $S \cup \P^{n_1+n_2-3} \cup \Bl_{pt} \P^{n_1-1}$, glued as follows. The third component is glued along its exceptional divisor to a linear $\P^{n_1-2} \subseteq \P^{n_1+n_2-3}$, while $S$ is glued along $C_2$ to a rational normal curve in a complementary $\P^{n_2-1} \subseteq \P^{n_1+n_2-3}$, passing through the intersection of both subspaces. At last, $C_1 \subseteq S$ is glued along the strict transform of a rational normal curve in $\P^{n_1-1}$ passing through the blown-up point.
\end{enumerate}
\end{teo}

\subsection{Structure of the paper}

We will start with a review of Bridgeland stability conditions: Section \ref{sec:stability} will be devoted to the general theory, while Section \ref{sec:stabsurf} will discuss the special case of surfaces. 

The next two sections form the technical heart of the paper. In Section \ref{sec:irrviaWC} we describe how to relate the semistable objects on both sides of a wall. In many cases, this allows us to describe the irreducible components of the moduli space on one side of the wall in terms of the other side. From here, in order to describe the gluing of these irreducible components, we need to understand the local structure of the moduli spaces.

The local structure of a moduli space $M_\sigma(v)$ at a point $[E]$ is deeply related to the deformation theory of $E$, which in turn is encoded in the differential graded Lie algebra $R\Hom(E, E)$. We will devote Section \ref{sec:localstr} to describe explicit representatives $R\Hom(E, E)$. This will give us an effective way to compute the local structure.

We will apply this discussion to describe moduli spaces arising from various stability conditions on surfaces. First, we will use the stability conditions of \cite{TX22} in Section \ref{sec:single}. These stability conditions are built out of the data of a contraction $S \to T$ of a single smooth, rational curve. After that, we will use the stability conditions from \cite{Vil25}, which are built out of the data of a contraction of multiple smooth, rational curves. For Section \ref{sec:disjoint} we will use a contraction of multiple disjoint curves, proving Theorem \ref{teo:intro_disjoint}. Finally, we will use a contraction of two intersecting curves in Section \ref{sec:inter}, which will prove Theorem \ref{teo:intro_inter}.

\subsection{Conventions}

We will work over the complex numbers. Given a smooth, projective surface $S$ and $E \in \DC^b(S)$, we let $\ch(E)$ to be the Chern character of $E$. Given $\beta \in \NS(S)_\R$, we set $\ch^\beta(E) = \ch(E).\exp(-\beta)$. We have $\ch_0^\beta(E) = \ch_0(E)$, $\ch_1^\beta(E) = \ch_1(E)-\beta.\ch_0(E)$, and $\ch_2^\beta(E) = \ch_2(E)-\beta.\ch_1(E) + \frac{\beta^2}{2}\ch_0(E)$.

If $A, B \in \DC^b(S)$, we will denote $\Ext^i(A, B) = \Hom(A, B[i])$. We have composition maps $\Ext^i(B, C) \times \Ext^j(A, B) \to \Ext^{i+j}(A, C)$, which we denote by $\circ$.  

At last, a \emph{vector bundle} $E$ on $X$ will be a locally free sheaf of constant rank $r$. If $\{U_i\}_{i \in I}$ is a trivializing cover, with $\alpha_i\colon \O_{U_i}^{\oplus r} \to E|_{U_i}$, we denote by $f_{ij} = \alpha_j^{-1} \circ \alpha_i$ the transition functions, so that the cocycle condition $f_{jk} \circ f_{ij} = f_{ik}$ holds.

\subsection{Acknowledgements}

First and foremost, I would like to thank my PhD advisor, Giulia Saccà. I am deeply grateful for many discussions in the last year, especially on deformation theory. More than that, her constant support has helped me on every step of this project.

I would like to thank Rafah Hajjar and Laura Pertusi for various discussions around this project. This project was done in parallel to \cite{Vil25}; as such, I am grateful to Arend Bayer and Tzu-Yang Chou for informing me of the related work \cite{Cho24}.

This work was partially supported from the Simons Foundation (grant number SFI-MPS-MOV-00006719-09), and by the NSF (grant number DMS-2052934).

\section{Bridgeland stability conditions} \label{sec:stability}

Let us start by recalling the definition of a stability condition, following \cite{BM23}*{\textsection 2.1}. To do so, we fix a smooth, projective variety $X$, a finite rank lattice $\Lambda$, and an homomorphism $v\colon K(\DC^b(X)) \to \Lambda$. A \emph{Bridgeland stability condition} on $X$ (with respect to $(\Lambda, v)$) is a pair $\sigma=(Z, \Pslicing)$ satisfying the following properties:
\begin{enumerate}[label=(\alph*)]
\item $\Pslicing$ is a \emph{slicing} of $\DC^b(X)$: a collection of full subcategories $\{ \Pslicing(\phi) \}_{\phi \in \R}$ subject to the following relations:
\begin{itemize}
\item For all $\phi$, we have $\Pslicing(\phi+1) = \Pslicing(\phi)[1]$.
\item Given $\phi_1>\phi_2$ and $E_i \in \Pslicing(\phi_i)$, we have $\Hom(E_1, E_2)=0$. 
\item For any $E \in \D$, there is a sequence of maps 
\[ 0=E_0 \xrightarrow{i_1} E_1 \xrightarrow{i_2} \dots \xrightarrow{i_m} E_m=E, \]
and real numbers $\phi_1>\dots > \phi_m$, such that the cone of $i_k$ is in $\Pslicing(\phi_k)$ for each $k$. The cones are called the \emph{Harder--Narasimhan factors} of $E$.
\end{itemize}
\item $Z\colon \Lambda \to \C$ is a $\Z$-linear map, called the \emph{central charge}. 
\end{enumerate}
We write $Z(E)=Z(v(E))$. We impose the following compatibility condition:
\begin{enumerate}[label=(\alph*), resume]
\item For each non-zero $E \in \Pslicing(\phi)$, we have $Z(E) \in \R_{>0}\cdot \exp(i\pi \phi)$.
\end{enumerate}
Finally, we add the following three extra properties;
\begin{enumerate}[label=(\alph*), resume]
\item There exists a quadratic form $Q$ on $\Lambda_\R$ such that (i) $Q$ is negative definite on $\ker Z$, and (ii) $Q(E) \geq 0$ for any $E \in \Pslicing(\phi)$. 
\item The property of being in $\Pslicing(\phi)$ is open in families over any base scheme.
\item For any $\phi \in \R$ and any $v \in \Lambda$, the collection of objects $E \in \Pslicing(\phi)$ with $v(E)=v$ is bounded.
\end{enumerate}

Properties (a)--(c) define a \emph{pre-stability condition}, and they constitute the original assumptions in \cite{Bri07}. Property (d) is known as the \emph{support property}, and it is key to getting a well-behaved wall-and-chamber structure. Lastly, properties (e)--(f) will give us the existence of moduli spaces, as we will review in Theorem \ref{teo:existsmoduli_main}. We denote by $\Stab(X)$ the collection of stability conditions on $X$.

\begin{teo}[Bridgeland deformation theorem, cf. \cite{BM23}*{Theorem 2.2}]
The space $\Stab(X)$ carries a natural topology, such that forgetful map $\Stab(X) \to \Hom(\Lambda, \C)$ is a local homeomorphism. In particular, $\Stab(X)$ carries the structure of a complex manifold. 
\end{teo}

\subsection{Moduli stacks and spaces} \label{subsec:existsmoduli}

Given a stability condition $\sigma = (Z, \Pslicing)$, a vector $v \in \Lambda$, and a phase $\phi$, we consider the assignment
\begin{equation} \label{eq:existsmoduli_stack}
Z \mapsto \{ \E \in \DC_{Z-perf}(Z \times X) : \forall z \in Z, \E|_z \in \Pslicing(\phi), v(\E|_z) = v \},
\end{equation}
where $\E|_z$ denotes the derived restriction of $\E$ to $\{z\} \times X$. This defines a subfunctor $\M_\sigma(v)$ of $\M_{pug}(X)$, the algebraic stack of perfect and universally gluable objects on $X$ from \cite{Lie06}.

\begin{teo}[cf. \cite{BM23}*{Theorem 2.3}] \label{teo:existsmoduli_main}
\begin{enumerate}
\item For each numerical class $v$, the subfunctor $\M_\sigma(v)$ defines an open substack of $\M_{pug}(X)$. In particular, $\M_\sigma(v)$ is an algebraic stack.
\item Moreover, $\M_\sigma(v)$ admits a proper good moduli space $M_\sigma(v)$. If $\M_\sigma(v)$ consists only on $\sigma$-stable objects, then $\M_\sigma(v) \to M_\sigma(v)$ is a $\mathbb{G}_m$-gerbe.
\end{enumerate}
\end{teo}

\begin{proof}
Part (1) is \cite{PT19}*{\textsection 4.4}. Part (2) is \cite{AHLH23}*{\textsection 7} for the existence of the proper good moduli space, and \cite{Lie06}*{\textsection 4.3} for the $\mathbb{G}_m$-gerbe. 
\end{proof}

\begin{remark} \label{remark:existsmoduli_twisted}
Note that if $M_\sigma(v)$ parametrizes only stable objects, then there is a (twisted) universal family $\U_{M_\sigma(v)} \in \DC^b(M_\sigma(v) \times X, \alpha \boxtimes 1)$, where $\alpha$ is a Brauer class on $M_\sigma(v)$. In fact, the $\mathbb{G}_m$-gerbe $\M_\sigma(v) \to M_\sigma(v)$ admits sections étale locally (cf. \cite{Alp25}*{Proposition 6.4.17(3)}).

This way, there is an étale cover $\{U_i \to M_\sigma(v)\}$ and sections $s_i\colon U_i \to \M_\sigma(v)$. By definition, the $\sigma_i$ correspond to objects $\E_i \in \DC_{U_i-perf}(U_i \times X)$. For each $i, j$, the restrictions of $s_i$ and $s_j$ to $U_i \times_{M_\sigma(v)} U_j$ differ by some $f_{ij} \in \mathbb{G}_m(U_i \times_{M_\sigma(v)} U_j)$, as $\M_\sigma(v) \to M_\sigma(v)$ is a $\mathbb{G}_m$-gerbe. The $f_{ij}$ define a Brauer class $\alpha$, and the $\E_i$ glue to an $\alpha \boxtimes 1$-twisted universal family. 

Note that in this case the closed points of $M_\sigma(v)$ correspond to isomorphism classes of $\sigma$-stable objects with numerical vector $v$. In general, the closed points of $M_\sigma(v)$ correspond to S-equivalence classes of $\sigma$-semistable objects. See \cite{AHLH23}*{Lemma 7.19}, and the discussion in \cite{AS25}*{\textsection 2.7}.
\end{remark}

\subsection{Walls and chambers}

Fix a vector $v \in \Lambda$. Given $\sigma \in \Stab(X)$, we consider the moduli space $M_\sigma(v)$ as before. As we vary $\sigma$, we want to understand how the moduli space changes. It turns out that this behaves in a controlled way.

\begin{teo}[\citelist{\cite{Bri08}*{9.3} \cite{BM11}*{3.3}}] \label{teo:wc_main}
Fix a connected component $\Stab_0(X)$ of $\Stab(X)$. Fix $\sigma_c \in \Stab_0(X)$, and let $\Cchamb$ be the set of $\sigma_c$-semistable objects of class $v$. There exists a collection $\{ W_u \} \subseteq \Stab(X)$ of closed real codimension one submanifolds with boundary satisfying the following.
\begin{enumerate}
\item The collection $W_u$ is locally finite. 
\item Let $C \subseteq \Stab_0(X)-\bigcup_u W_u$ be a connected component of the complement. Then $E \in \Cchamb$ is $\sigma$-semistable for \emph{some} $\sigma \in C$ if and only if it is $\sigma$-semistable for \emph{all} $\sigma \in C$.
\item Each $W_u$ is contained in the inverse image of
\begin{equation} \label{eq:wc_wallformula}
\{Z\in \Hom(\Lambda, \C) : \Re Z(u) \cdot \Im Z(v_0) = \Re Z(v_0) \cdot \Im Z(u)\}
\end{equation}
under the forgetful map $\Stab_0(X) \to \Hom(\Lambda, \C)$. 
\end{enumerate}
Furthermore, assume that $v$ is primitive.
\begin{enumerate}[resume]
\item Given $C$ as in (2), we have that $E \in \Cchamb$ is stable for some $\sigma \in C$ if and only if it is stable for all $\sigma \in C$.
\item For every $\sigma=(Z, \Pslicing) \in W_u$, there exists a phase $\phi$, an object $E \in \Cchamb$ of phase $\phi$, and some $F_u \in \Pslicing(\phi)$ with $v(F_u)=u$ and with an inclusion $F_u \hookrightarrow E$ in the category $\Pslicing(\phi)$.
\end{enumerate}
\end{teo}

We call the $\{W_u\}$ the \emph{walls} corresponding to the vector $v$. Using Theorem \ref{teo:wc_main} we can describe the walls passing through $\sigma_0$ by computing the $\sigma_0$-semistable objects of numerical class $v$, and then finding their semistable factors.

\section{Stability conditions on surfaces} \label{sec:stabsurf}

As we mentioned in the introduction, we are mostly interested in stability conditions on surfaces. In this section we will review various results about their existence and the collection of semistable objects whose numerical class is that of a point.

\subsection{Arcara--Bertram}

Our starting point is the following result of about existence of stability conditions on (smooth, projective) surfaces, due to Arcara--Bertram. We point out that for K3 surfaces a similar construction was performed by Bridgeland in \cite{Bri08}. This has also been adapted to normal surfaces in \cite{Lan24}.

Let us fix some notation. Fix a smooth, projective surface $S$, and set $\Lambda=K^{\text{num}}(S)$. Given $\beta \in \NS(S)_\R$ and $\omega \in \Amp(S)_\R$, we let 
\begin{align*}
\Ttors_{\beta, \omega} &= \langle \{ E: E \text{ torsion} \} \cup \{E: E \text{ torsion-free, stable, } \mu(E)>\beta.\omega \} \rangle, \\
\Ftors_{\beta, \omega} &= \langle \{ E: E \text{ torsion-free, stable, }\mu(E)\leq \beta.\omega\} \rangle.
\end{align*}
These subcategories of $\Coh(S)$ define a \emph{torsion pair}, cf. \cite{AB13}*{p. 6}. We let $\A_{\beta, \omega} \subseteq \DC^b(S)$ be the associated tilt (cf. \cite{AB13}*{p.5}).

\begin{teo}[\citelist{\cite{AB13} \cite{Tod08}*{\textsection 4}}]
Let $S$ be a smooth, projective surface. Given $\beta \in \NS(S)_\R, \omega \in \Amp(S)_\R$, there exists a stability condition $\sigma_{\beta, \omega} =(Z_{\beta, \omega}, \Pslicing_{\beta, \omega})$, with
\begin{equation} \label{eq:AB_ccharge}
Z_{\beta, \omega}(E) = -\ch_2^\beta(E) +\frac{\omega^2}{2}\ch_0(E) + i\omega.\ch_1^\beta(E),
\end{equation}
and $\Pslicing_{\beta, \omega}((0, 1])=\A_{\beta, \omega}$. Moreover, the assignment $(\beta, \omega) \mapsto \sigma_{\beta, \omega}$ defines a continuous map $\Sigma\colon \NS(S)_\R \times \Amp(S)_\R \to \Stab(S)$. 
\end{teo}

\begin{prop}[\cite{AB13}*{p. 8}] \label{prop:AB_moduli}
Let $S, \beta, \omega$ be as above. The $\sigma_{\beta, \omega}$-semistable objects of phase 1 and numerical vector $[pt]$ consists of the skyscraper sheaves $\{ \O_x: x \in S \}$. All of them are stable, and $M_{\sigma_{\beta, \omega}}([pt]) \cong S$. 
\end{prop}

\subsection{Limits}

Given a (smooth, projective) surface $S$, the Arcara--Bertram construction provides plenty of stability conditions, parametrized by $\beta \in \NS(S)_\R$ and $\omega \in \Amp(S)_\R$. Producing more explicit examples (besides composing with the $\Aut(\DC^b(X))$ and $\tilde{\GL}_2^+(\R)$ actions) seems to be a difficult question. 

A natural approach is trying to describe stability conditions that lie in the closure of the set $\{\sigma_{\beta, \omega}: \beta \in \NS(S)_\R, \omega \in \Amp(S)_\R\}$. By continuity, such stability conditions will have central charge $Z_{\beta, \lambda}$ as in \eqref{eq:AB_ccharge}, but with $\lambda \in \Nef(S)$. We will focus on the case when $\lambda=f^\ast \eta$, where $f\colon S \to T$ is a birational morphism to a normal, projective surface $T$ and $\eta \in \Amp(T)_\R$.

\begin{question} \label{question:limits_main}
Let $f\colon S \to T$ be a birational morphism from a smooth, projective surface $S$ to a normal, projective surface $T$. Let $\beta \in \NS(S)_\R$ and $\eta \in \Amp(T)_\R$ be given. Is there a stability condition $\overline{\sigma}_{\beta, f^\ast\eta} \in \Stab(S)$ with central charge $Z_{\beta, f^\ast\eta}$ such that
\[ \lim_{\omega \in\Amp(S), \omega \to f^\ast\eta} \sigma_{\beta, \omega} = \overline{\sigma}_{\beta, f^\ast\eta} \]
in the topology of $\Stab(S)$?
\end{question}

Question \ref{question:limits_main} has been studied extensively as an attempt to produce new examples of Bridgeland stability conditions on surfaces. In the next theorem we have collected various results proven in this direction, including the recent \citelist{\cite{Cho24} \cite{Vil25}}.

\begin{teo} \label{teo:limits_all}
Let $S$ be a smooth, projective variety and let $f\colon S \to T$ be a birational morphism to a normal, projective surface $T$. Let $\beta \in \NS(S)_\Q$ and $\eta \in \Amp(T)_\Q$. Then, Question \ref{question:limits_main} has a positive answer in the following cases:
\begin{enumerate}
\item If $S$ is a K3 surface, $f$ is crepant, and $\beta\in \NS(S)_\Q$ satisfies $ \beta.(\sum a_i C_i) \notin \Z$ for any $C_i \in \Exc(f)$ and any $a_i \in \Z$ with $(\sum a_iC_i)^2 = -2$. 
\item If $f$ is the contraction of a $(-1)$-curve and $\beta= 0$. 
\item If $f$ is the contraction of a $(-n)$-curve $C$ and $\beta.C + n/2 \notin \Z$.
\item If $f$ is crepant, $T$ only has one singularity, and $\beta$ satisfies $\beta.C_i >0$ for all $C_i \in \Exc(f)$, $\beta.f^\ast\eta=0$, and $\beta.\sum a_i C_i=0$ for the fundamental cycle of the singularity. 
\item If each irreducible component of $\Exc(f)$ is a chain of rational curves $C_{i,1} \cup \dots \cup C_{i,r_i}$ with no $(-1)$-curves intersecting other curves in $\Exc(f)$, and no $(-2)$-curves intersecting more then one other curve in $\Exc(f)$; and $\beta \in \NS(S)_\Q$ satisfying the conditions:
\begin{itemize}
\item There are $k_{i,j}\in \Z$ with $k_{i,j}-1<\beta.C_{i,j}+C_{i,j}^2/2 < k_{i,j}$ for all $i,j$,
\item $\beta.(C_{i,j}+\dots + C_{i,j'})+(C_{i,j}^2+\dots+C_{i,j'})/2 < (k_{i,j}+\dots+k_{i,j'})-(j-j')$ for all $i$ and all $j\leq j'$.
\end{itemize}
\end{enumerate}
\end{teo}

\begin{proof}
Part (1) follows directly from \cite{Bri08}*{Theorem 1.1}. Part (2) is a consequence of \cite{Tod13}*{Theorem 1.2}. Part (3) is \cite{TX22}*{Theorem 5.4}. Part (4) is \cite{Cho24}*{Theorem 1.1} (cf. \cite{LR22}*{Remark 1.3}). Lastly, part (5) is \cite{Vil25}*{Theorem 1.3}.
\end{proof}

Let us point out some relations among parts (1)--(5) in Theorem \ref{teo:limits_all}. First, parts (2) and (3) are included in the cases covered by (5). Part (4) generalizes (1) when $T$ has a single singularity. Lastly, part (5) generalizes (1) and (4) when $T$ has only $A_n$ singularities and for $\beta$ satisfying the conditions in (5).

Let us finish up this section by describing the semistable objects of the stability conditions in part (5) of Theorem \ref{teo:limits_all}. Compare this to Proposition \ref{prop:AB_moduli}.

\begin{prop} \label{prop:limits_ssfactors}
Let $f\colon S \to T$, $\beta \in \NS(S)_\Q$ and $\eta \in \Amp(T)_\Q$ be as in part (5) of Theorem \ref{teo:limits_all}. Denote by $\overline{\sigma}_{\beta, f^\ast\eta} \in \Stab(S)$ the corresponding stability condition. Also, fix $C_{i,j}$ and $k_{i,j}$ as in the theorem.

Given $x \in S \setminus \Exc(f)$, we have that $\O_x$ is $\overline{\sigma}_{\beta, f^\ast\eta}$-stable. Otherwise, if $x \in C_{i,1} \cup \dots \cup C_{i,r_i}$, then the $\overline{\sigma}_{\beta, f^\ast\eta}$-stable factors of $\O_x$ are $\O_{C_{i,1}\cup \dots \cup C_{i,r_i}}(k_{i,1}, \dots, k_{i,r_i})$, and $\{\O_{C_{i,j}}(k_{i,j}-1)[1]\}_{j=1, \dots, r_i}$.
\end{prop}

\begin{proof}
Let $0 = E_0 \subset \dots E_r = \O_x$ be a filtration of $\O_x$ via stable factors. Consider the exact sequence $0 \to E_{r-1} \to \O_x \to E_r/E_{r-1} \to 0$. Here, $H^0(E_{r-1})$ and $H^0(E_r/E_{r-1})$ are torsion, thanks to \cite{Vil25}*{Lemma 3.13}. This way, a long exact sequence argument shows that $H^{-1}(E_r/E_{r-1})$ is torsion as well. Inductively, one quickly shows that $H^i(E_j/E_{j-1})$ are all torsion.

This way, note that the stable factors of $\O_x$ must have support on $\Exc(f)$ union finitely many points. By looking at the magnitude of $\Re Z_{\beta, f^\ast\eta}$ we get that no other skyscraper sheaf can be a factor of $\O_x$. So it suffices to classify the objects of phase 1 whose support is one-dimensional. 

To do so, note that \cite{Vil25}*{Lemma 5.8} implies that all stable objects of phase 1 whose support is one dimensional will be of the form $E=\O_{C_{i,j, j'}}(d_j, \dots, d_{j'})$ or $E=\O_{C_{i,j, j'}}(d_j, \dots, d_{j'})[1]$, where $C_{i,j,j'} = C_{i,j} \cup C_{i,j+1} \cup \dots \cup C_{i,j'}$. 
\begin{itemize}
\item In the first case, note that $d_a \geq k_{i,a}$ for all $j \leq a \leq j'$. In fact, we have that $\Ttors$ is closed under quotients. But then $\O_{C_{i,j,j'}}(k_{i,j}, \dots, k_{i,j'})$ is in $\Ttors$ (by \cite{Vil25}*{Lemma 5.10}), and injects in $E$. 

\item In the second case, note that we have $d_a < k_{i,a}$ for some $j \leq a \leq j'$ by \cite{Vil25}*{Lemma 5.10}. Then the surjection $E \to \O_{C_{i,a}}(d_a)$ will lead to a contradiction unless $j=j'=a$. At last, if $d_a < k_{i, a}-1$, we reach a contradiction directly from the short exact sequence 
\[ 0 \to \O_x \to \O_{C_{i,a}}(d_a)[1] \to \O_{C_{i,a}}(d_a+1)[1] \to 0. \]
\end{itemize}
This proves the required characterization.
\end{proof}

\section{Irreducible components via wall-crossing} \label{sec:irrviaWC}

Let us fix some notation. Fix some ambient smooth, projective variety $X$, and a vector $v \in \Lambda$. Fix a wall in $\Stab(X)$ for $v$, corresponding to a decomposition $v = u + w$. Fix also stability conditions $\sigma_0$ on the wall, and $\sigma_{\pm}$ on each side of the wall. Denote by $M_0 = M_{\sigma_0}(v)$ the moduli space of $\sigma_0$-semistable objects with class $v$, and $M_u = M_{\sigma_0}(u), M_w$ defined in a similar way.

The wall-and-chamber decomposition of Theorem \ref{teo:wc_main} gives us a procedure to relate the moduli spaces $M_{\pm} := M_{\sigma_{\pm}}(v)$, following ideas of \citelist{\cite{AB13} \cite{BM14} \cite{Xia18} \cite{TX22}}. We assume that the moduli space $M_+$ is known, and that we want to describe $M_-$. For each irreducible component of $M_+$, we look at the locus admitting maps from objects with numerical class $[u]$. We will identify the objects of this locus and replace them by extensions in the opposite direction. 

The extensions mentioned above can be described directly, by constructing a universal extension over $M_\sigma(u) \times M_\sigma(w)$. Under some extra assumptions, this will yield a scheme with a map to $M_-$. The image of this map corresponds to those $\sigma_-$-semistable objects that are strictly semistable on the wall. These objects, together with those obtained from the irreducible components of $M_+$, will allow us to describe the irreducible components of $M_-$.

\subsection{A bundle of extensions} \label{subsec:bundle}

Let us start with the following observation.

\begin{remark}\label{remark:bundle_heart}
If $\sigma_0$ and $\sigma_{\pm}$ have the same imaginary part (and are close to $\sigma_0$), then the hearts $\Pslicing_0((0, 1])$ and $\Pslicing_{\pm}((0, 1])$ will agree. In general, one can use the $\widetilde{\GL}_2^+$ action to replace $\sigma_{\pm}$ for other stability conditions in the same chamber such that there is a common heart $\A = \Pslicing_0((\psi, \psi+1]) = \Pslicing_{\pm}((\psi, \psi+1])$ for some $\psi$. 
\end{remark}

Fix a heart $\A$ as above. Let us assume the following
\begin{enumerate}
\item[\textbf{A1.}] The stability condition $\sigma_0$ lies only on one $v$-wall, corresponding to the decomposition $[v]=[u]+[w]$
\item[\textbf{A2.}] The stability condition $\sigma_0$ does not lie on a $u$-wall or on a $w$-wall, and $M_u:= M_{\sigma_0}(u)$, $M_w:=M_{\sigma_0}(w)$ only parametrize stable objects. 
\end{enumerate}
Condition \textbf{A1} is not too restrictive, as we can usually cross one wall at a time. On the other hand, condition \textbf{A2} is a bit more delicate (cf. \cite{AB13}*{pp. 24--25}). Using \textbf{A2}, we may assume that $\sigma_0$ and $\sigma_{\pm}$ lie in the same $u$-chamber and $w$-chamber. 

Let us add a third nice assumption.
\begin{enumerate}
\item[\textbf{A3.}] The vector $v$ is primitive. Hence, the moduli spaces $M_\pm$ only parametrize stable objects. 
\end{enumerate}
As we mentioned previously, our goal is to construct various schemes that parametrize $\sigma_-$-stable objects of numerical vector $v$. 

In this subsection we will look at extensions of the form $0 \to G \to E \to F \to 0$, where $[F] \in M_u$ and $[G] \in M_w$. We consider the space $M_u \times M_w \times X$ and the three projections $\rho\colon M_u \times M_w \times X \to M_u \times M_w$, $p_u\colon M_u \times M_w \times X \to M_u \times X$, and $p_w\colon M_u \times M_w \times X \to M_w \times X$. Consider
\[ \E = R\rho_\ast R\HOM(p_u^\ast \U_u, p_w^\ast \U_w) \in \DC^b(M_u \times M_w), \]
where $\U_u \in \DC^b(M_u \times X)$ and $\U_w \in \DC^b(M_w \times X)$ denote the respective (twisted) universal families\footnote{In general, a universal family only exist up to twisting by a Brauer class, cf. Remark \ref{remark:existsmoduli_twisted}. For the sake of clarity we are dropping this twist in our notation.}.

\begin{lemma}
Let $x \in M_u$ and $y \in M_w$ correspond to the objects $F$ and $G$. Then the (derived) restriction of $\E$ to $(x, y) \in M_u \times M_w$ is isomorphic to $R\Hom_X(F, G)$. 
\end{lemma}

\begin{proof}
Consider the diagram
\[ \begin{tikzcd} (x, y) \times X \arrow[r, hook, "i"] \arrow[d, "\rho"'] & M_u \times M_w \times X \arrow[d, "\rho"] \\ (x, y) \arrow[r, hook, "i"'] & M_u \times M_w. \end{tikzcd} \]
Note that the vertical maps are flat. In particular, we get that $Li^\ast R\rho_\ast = R\rho_\ast Li^\ast$. The result $Li^\ast \E = R\Hom(E, G)$ now follows directly.
\end{proof}

\begin{cor}
We have that $\E \in \DC^b(M_u \times M_w)$ has $H^j(\E)=0$ for $j \leq 0$, and $i^\ast H^1(\E) = H^1(\E|_{(x, y)}) = \Ext^1(F, G)$, where the $i^\ast$ denotes the non-derived restriction. 
\end{cor}

\begin{proof}
The first part follows by the fact that $\Ext^j(F, G) = 0$ for $j <0$, and also for $j=0$ (as we are assuming $u \neq w$). The second part is now clear, e.g. by the spectral sequence $E_2^{p, q} = L^p i^\ast H^q(\E) \Rightarrow H^{p+q}(Li^\ast \E)$. 
\end{proof}

In other words, the sheaf $\HC^1(\E)$ carries all the extensions of objects in $M_u$ with objects in $M_w$. We will add some assumptions to ensure that $\HC^1(\E)$ is a vector bundle on its support.
\begin{enumerate}
\item[\textbf{A4.}] Both $M_u$ and $M_w$ are reduced. 
\end{enumerate}

\begin{cor}
Let $Z \subset M_u \times M_w$ be a reduced closed subvariety where the rank of $\Ext^1(F, G)$ remains constant. Then the (non-derived) restriction of $H^1(\E)$ to $Z$ is locally free. In particular, we can take $Z=M_u \times M_w$ if the rank of $H^1(\E)$ is constant, by \textbf{A4}.
\end{cor}

\begin{proof}
Follows directly from the fact that $H^1(\E)$ has constant rank on a reduced scheme.
\end{proof}

The rank of $H^1(\E)$ might not be constant, in which case the wall-crossing becomes more involved. We will add an assumption under which the destabilized loci are well-behaved.
\begin{enumerate}
\item[\textbf{A5.}] The rank of $H^1(\E)$ is constant on its support. In other words, there is a number number $r>0$ such that $H^1(\E)$ has either rank $0$ or $r$ at each point.
\end{enumerate}
In that case we take $Z=\Supp H^1(\E) \subset M_u \times M_w$, with its reduced structure. Here the (non-derived) restriction $\overline{\E}_Z = i_Z^\ast H^1(\E)$ is a vector bundle of rank $r$. We let $P = \P_Z(\overline{\E}_Z)$ to be its projectivization, and $h\colon P \to Z$ the bundle map.

\begin{remark}
Our arguments will apply similarly if we replace \textbf{A5} by
\begin{enumerate}
\item[\textbf{A5'.}] The rank of $H^1(\E)$ is constant on each connected component of its support.
\end{enumerate} 
In the cases we are interested though, the support of $H^1(\E)$ will be connected, hence these two conditions will give the same results. 
\end{remark}

\begin{prop}[cf. \cite{Xia18}*{4.3}]  \label{prop:bundle_family}
Let $j\colon Z \to M_u \times M_w$ be the inclusion map. We have a universal extension
\[ h^\ast Lj^\ast p_w^\ast \U_w \otimes \O_h(1) \to \U_P \to h^\ast Lj^\ast p_u^\ast \U_u \to h^\ast Lj^\ast p_w^\ast \U_w \otimes \O_h(1)[1] \]
in $\DC^b(P \times X)$. On each point $p \in P$, the restriction of this sequence corresponds to $0 \to G \to E \to F \to 0$, where $(F, G)$ and the extension class are determined by $p$.
\end{prop}

\begin{proof}
We have that
\[ \Hom_{P \times X}^1(h^\ast Lj^\ast p_u^\ast \U_u, h^\ast L j^\ast p_w^\ast \U_w \otimes \O_h(1)) = \Hom_Z^1(\overline{\E}_Z, Lj^\ast \E). \]
We now take the canonical map $\HC^1(\E)[-1] \to \E$ on $\DC^b(M_u \times M_w)$ and restrict it to $Z$. This gives us the claimed map. We point out that $p$ determines $h(p) =([F], [G]) \in M_u \times M_w$, and a class in $\P \Ext^1(F, G)$.
\end{proof}

\subsection{Embedding the bundle} \label{subsec:embbundle}

So far, we have constructed a projective bundle $P$ parametrizing extensions of the form $0 \to G \to E \to F \to 0$ in $\DC^b(X)$. In particular, we proved in Proposition \ref{prop:bundle_family} the existence of a family of extensions $\U_P \in \DC^b(P \times X)$.

\begin{claim}
Each object parametrized by $P$ is $\sigma_+$-stable. This follows directly from \textbf{A1}--\textbf{A2}.
\end{claim}

This way, the family $\U_P$ defines a map $P \to M_+$. We should think of its image as the locus of objects that are destabilized after crossing the wall. Showing that $P \to M_+$ is a closed embedding requires some work.

\begin{lemma}[cf. \cite{Xia18}*{4.5(2)}]
The map $P \to M_+$ is injective on closed points.
\end{lemma}

\begin{proof}
Let $p_1, p_2 \in P$ be two points whose image inside $M_+$ agrees. Denote by $E_1 \cong E_2$ the two extensions. We consider the diagram
\[ \begin{tikzcd} 0 \arrow[r] & G_1 \arrow[r] & E_1 \arrow[r] \arrow[d, "\cong"] & F_1 \arrow[r] & 0 \\ 0 \arrow[r] & G_2 \arrow[r] & E_2 \arrow[r] & F_2 \arrow[r] & 0. \end{tikzcd} \]
Note that the induced map $G_1 \to F_2$ is zero, hence we can fill the diagram with non-zero maps $G_1 \to G_2, F_1 \to F_2$. But then these have to be isomorphisms, as they are maps between stable objects. Up to a constant, this implies that the extension classes are the same, hence $p_1=p_2$. 
\end{proof}

The next step towards proving $P \to M_+$ is a closed embedding requires looking at the tangent spaces. To do so, we consider the following lemma.

\begin{lemma}[cf. \cite{Xia18}*{4.7}]
Let $p \in P$ be given, corresponding to an extension $0 \to G \xrightarrow{\alpha} E \xrightarrow{\beta} F \to 0$. Then the map $T_{P, p} \to T_{M_+, p}$ is injective. Moreover, using the Kodaira--Spencer map we get that the composition
\[ T_{P, p} \to T_{M_+, p} \xrightarrow[KS]{\cong} \Ext^1(E, E) \to \Ext^1(G, F) \]
is zero. 
\end{lemma}

\begin{proof}
Let $\eta \in T_{P, p}$ be a tangent vector that is mapped to zero in $T_{M_+, p}$. The Kodaira--Spencer map shows that the associated map $E \to E[1]$ vanishes.

On the other hand, we can consider the tangent vector $\eta$ and pushforward it to $M_u \times M_v$. The Künneth formula gives us $\eta_u \in T_{M_u, [F]}$, $\eta_w \in T_{M_w, [G]}$. Using the Kodaira--Spencer map on $M_u$ and $M_w$, these fit into the diagram
\[ \begin{tikzcd} G \arrow[r] \arrow[d, "\eta_w"'] & E \arrow[r] \arrow[d, "0"] & F \arrow[d, "\eta_u"] \\ {G[1]} \arrow[r] & {E[1]} \arrow[r] & {F[1].} \end{tikzcd} \]
Here, a fast long exact sequence argument shows that $\eta_u, \eta_w=0$, as there are no maps from $G$ to $F$. Thus, the tangent vector $\eta$ is vertical with respect to the map $h\colon P \to Z$, i.e. it is a tangent vector to the fiber $\P \Ext^1(F, G)$. Let us abuse the notation and write $\eta \in T_{\P \Ext^1(F, G), p}$.

To conclude, let us investigate the map $\Ext^1(F, G) \to \Ext^1(E, E)$ given by composition with $\alpha$ and $\beta$. Let $\gamma \in \Ext^1(F, G)$ be the extension class associated to $p$. We have the following diagram with exact rows:
\[ \begin{tikzcd} & \Ext^0(F, F) \arrow[r, "\gamma \circ -"] \arrow[d, "-\circ \beta"] & \Ext^1(F, G) \arrow[r, "\alpha \circ -"] \arrow[d, "-\circ \beta"] & \Ext^1(F, E) \arrow[d, "-\circ \beta"] \\ \Ext^0(E, E) \arrow[r, "\beta \circ -"'] & \Ext^0(E, F) \arrow[r, "\gamma \circ -"'] & \Ext^1(E, G) \arrow[r, "\alpha \circ -"'] & \Ext^1(E, E). \end{tikzcd} \]
The bottom map $\Ext^0(E,E) \to \Ext^0(E, F)$ is an isomorphism, and so the map $\Ext^1(E, G) \to \Ext^1(E, E)$ is injective. This way, the kernel of $\Ext^1(F, G) \to \Ext^1(E, E)$ is the image of $\Ext^0(G, G) \to \Ext^1(F, G)$, namely $\C \gamma$. In other words, the induced map $\Ext^1(F, G)/\C\gamma \to \Ext^1(E, E)$ is injective. But the left hand side is exactly $T_{\P \Ext^1(F, G), p}$ after identifications. This proves that $T_{P, p} \to T_{M_+, p}$ is injective, as required.

At last, let us show that the map $T_{P, p} \to \Ext^1(G, F)$ is zero. Let $\eta \in T_{P, p}$ be a tangent vector, and let $\eta_u, \eta_w$ be induced tangent vectors to $F, G$ as before. We get a diagram
\[ \begin{tikzcd} G \arrow[r, "\alpha"] \arrow[d, "\eta_w"'] & E \arrow[r, "\beta"] \arrow[d, "\eta"] & F \arrow[d, "\eta_u"] \\ {G[1]} \arrow[r] & {E[1]} \arrow[r] & {F[1],} \end{tikzcd} \]
where we once again used the Kodaira--Spencer map to identify the tangent vectors with extension classes. This way, we get $\beta \circ \eta \circ \alpha = \beta \circ \alpha \circ \eta_w = 0$, as claimed.
\end{proof}

\begin{cor}[cf. \cite{Xia18}*{4.7}] \label{cor:embbundle_tangentmap}
We have the following commutative diagram, where $K_p$ is the kernel of the map $\Ext^1(E, E) \to \Ext^1(G, F)$.
\begin{equation} \label{eq:bundle_diagram}
\begin{tikzcd}
0 \arrow[r] & T_{P, p} \arrow[r] \arrow[d] & T_{M_+, p} \arrow[r] \arrow[d, "\cong"', "KS"] & T_{M_+, p}/T_{P, p} \arrow[r] \arrow[d] & 0 \\
0 \arrow[r] & K_p \arrow[r] & \Ext^1(E, E) \arrow[r] & \Ext^1(G, F).
\end{tikzcd}
\end{equation}
\end{cor}

\begin{cor}
The map $P \to M_+$ is a closed embedding.
\end{cor}

\begin{proof}
We have that $P \to M_+$ is injective on closed points and on tangent vectors. We also have that $P \to M_+$ is proper (as $P$ and $M_+$ are proper), hence it is finite (as it has finite fibers). To prove that the map is a closed embedding, we can work locally on $M_+$. The result follows now from the following algebraic lemma.
\end{proof}

\begin{lemma} \label{lemma:embbundle_closed}
Let $R, S$ be finite type $\C$-algebras, and let $\phi\colon S \to R$ be a ring homomorphism. Assume that the induced map $\Spec R \to \Spec S$ is injective on closed points and on tangent vectors. Then $S \to R$ is surjective, and so $\Spec R \to \Spec S$ is a closed embedding.
\end{lemma}

\subsection{Elementary modification} \label{subsec:em}

In the previous subsection we were able to identify the objects inside $M_+$ that are destabilized by the wall-crossing, as the ones parametrized by $P$. The question now is: how do we replace this locus with $\sigma_-$-stable objects? 

Informally, the strategy is to blow-up $P \subset M_+$, and modify the family to get $\sigma_-$-stable objects. This is a bit delicate, as in general $M_+$ might not be smooth. In the cases we are interested, it turns out that each irreducible component of $M_+$ will be smooth. So instead we will work on each irreducible component of $M_+$.

This way, we fix an irreducible component $M_+^i$ of $M_+$, and we let $P^i =P \cap M_+^i$ be the set-theoretic intersection, endowed with its reduced structure. Let us assume the (last!) condition.
\begin{enumerate}
\item[\textbf{A6.}] The space $P^i$ is smooth, and $M_+^i$ is smooth along $P^i$. 
\end{enumerate}
In this case, the diagram \eqref{eq:bundle_diagram} from Corollary \ref{cor:embbundle_tangentmap} gives us:
\begin{equation} \label{eq:em_tangentmap}
\begin{tikzcd}
0 \arrow[r] & T_{P^i, p} \arrow[r] \arrow[d] & T_{M_+^i, p} \arrow[r] \arrow[d, hook', "KS"] & N_{P_i/M_+^i, p} \arrow[r] \arrow[d] & 0 \\
0 \arrow[r] & K_p \arrow[r] & \Ext^1(E, E) \arrow[r] & \Ext^1(G, F).
\end{tikzcd}
\end{equation}
We point out that the middle map might not be an isomorphism anymore, as we are dealing only with an irreducible component of $M_+$. In any case, it is still injective. 

We consider the blow-up $\Bl_{P^i} M_+^i$ of $M_+^i$ along $P^i$. Note that the exceptional divisor is the projectivization of the normal bundle $N_{P^i/M_+^i}$, which is now a locally free sheaf on $P^i$. In particular, giving a point of the exceptional divisor is the same datum as giving a point $p \in P^i$, together with a non-zero element of $N_{P^i/N_+^i, p}$ (up to a constant). We have the following diagram:
\[ \begin{tikzcd} \P(N_{P^i/M_+^i}) \arrow[r, hook, "d"] \arrow[d, "c"'] & \Bl_{P^i} M_+^i \arrow[d, "b"] \\ P^i \arrow[r, hook] & M_+^i. \end{tikzcd} \]

Let us take $Lb^\ast \U_{M_+^i}$, the pullback of the universal family of $M_+^i \times X$. The restriction via $d$ is isomorphic to the pullback of the family $c^\ast \U_{P^i}$ on $P^i\times X$, up to a line bundle on $P^i$: $Ld^\ast Lb^\ast \U_{M_+^i} \cong c^\ast (\U_{P^i} \otimes \rho^\ast \L)$. By the construction in Proposition \ref{prop:bundle_family}, the right hand side admits a map to $(h^\ast  Lj_Z^\ast \pi_A^\ast \U_A \otimes \rho^\ast \L)|_{P^i}$. Set $\K \in \DC^b(\Bl_{P^i}M_+^i)$ to be the object fitting in the triangle
\begin{equation} \label{eq:em_modification}
\K \to Lb^\ast \U_{M_+^i} \to d_\ast \left( (h^\ast  Lj_Z^\ast \pi_A^\ast \U_A \otimes \rho^\ast \L)|_{P^i}\right) \to \K[1].
\end{equation}

\begin{prop}[cf. \cite{Xia18}*{4.14}] \label{prop:em_main}
The object $\K \in \DC^b(\Bl_{P^i} M_+^i \times X)$ is a flat family of objects with vector $[v]$. If $x \in \Bl_{P^i} M_+^i$ is not in the exceptional divisor of $b$, the associated object $\K_x$ equals the object $\U_{M_i^+}|_{b(x)}$. Otherwise, if $x=d(y)$ is in the exceptional divisor (with $y \in \P(N_{P^i/M_i^+})$), then $\K_x$ fits into a triangle $F \to \K_x \to G \xrightarrow{\xi} F[1]$, where $F, G$ are determined by $c(y)$, and $\xi$ is (up to a constant) the image of $x$ via the map $N_{P^i/M_+^i, c(y)} \to \Ext^1(G, F)$. 
\end{prop}

\begin{proof}
Let us look at the objects $\K_x \in \DC^b(X)$, for various points $x \in \Bl_{P^i} M_+^i$. If $x$ is not in the exceptional divisor, then \eqref{eq:em_modification} restricts to $\K_x \cong \U_{M_i^+}|_{b(x)}$. 

For the other case, let us denote $p=c(y) = c(d^{-1}(x))$, the point of $P^i$ where $x$ lies over. The triangle \eqref{eq:em_modification} restricts to
\[ \K_x \to \U_{M_+^i, p} \xrightarrow{\beta} \left( (h^\ast  Lj_Z^\ast \pi_A^\ast \U_A \otimes \rho^\ast \L)|_{P^i}\right)|_p \to \ast. \]
Let us take cohomology with respect to the heart $\A$ (cf. Remark \ref{remark:bundle_heart}). For the middle object we clearly get that $\U_{M_+^i, p} = E$, the object in $\DC^b(X)$ parametrized by $p$. For the rightmost object we have
\[ \HC_\A^0 \left( (h^\ast  Lj_Z^\ast \pi_A^\ast \U_A \otimes \rho^\ast \L)|_{P^i}\right)|_p = \HC_\A^{-1} \left( (h^\ast  Lj_Z^\ast \pi_A^\ast \U_A \otimes \rho^\ast \L)|_{P^i}\right)|_p = F, \]
and zero otherwise, and so we get a triangle 
\[ F[1] \to \left( (h^\ast  Lj_Z^\ast \pi_A^\ast \U_A \otimes \rho^\ast \L)|_{P^i}\right)|_p \to F \to \ast. \]
Following the identifications, the map $\beta\colon E \to \left( (h^\ast  Lj_Z^\ast \pi_A^\ast \U_A \otimes \rho^\ast \L)|_{P^i}\right)|_p$ is now given by the two maps $E \to F$ and $\xi$. At last, by taking the long exact sequence of cohomology objects with respect to $\A$, and using that $\ker(E \to F) = G$, we get the required result.
\end{proof}

Let us point out two small subtleties of Proposition \ref{prop:em_main}. First, it might be the case that the objects $\K_x$, for $x$ in the exceptional divisor, are in fact split extensions. This happens if $N_{P^i/M_+^i, c(y)} \to \Ext^1(G, F)$ is not injective. Second, it is possible that $\K_x \cong \K_{x'}$ for $x \neq x'$ in the exceptional divisor.

To address the first part, we need a way to check whether the map $N_{P^i/M_+^i} \to \Ext^1(G, F)$ is injective. In applications, this can be done by considering the rest of the diagram \eqref{eq:em_tangentmap}. 

The second point is a bit more delicate. It is easy to see that if $N_{P_i/M_+^i, p} \to \Ext^1(G, F)$ is injective, then the extension classes $\{\K_x \}_{x \in d(c^{-1}(p))}$ are pairwise non-isomorphic. However, it might well be the case that $\K_x \cong \K_{x'}$ for some $x, x'$ with $c(x) \neq c(x')$; cf. \cite{AB13}. Regardless, we get the following result.

\begin{cor} \label{cor:em_newfamily}
Assume that for all $p \in P^i$ the map $N_{P^i/M_+^i, p} \to \Ext^1(G, F)$ is injective. Then the object $\K \in \DC^b(\Bl_{P^i} M_+^i \times X)$ defines a map $\Bl_{P^i} M_+^i \to M_-$.
\end{cor}

\subsection{A criterion for isomorphisms} \label{subsec:critiso}

Using the constructions from Subsections \ref{subsec:bundle} and \ref{subsec:em}, we are able to produce various closed subschemes of $M_-$: the image of the bundle $P_-$, and the elementary modifications $\tilde{M}_-^i \to M_-$ respectively. Of course, these closed subschemes are not disjoint in general: each $\tilde{M}_-^i$ will intersect $P-_-$, say at a closed subscheme $Z_i$. However, the images of $P_-$ and $\tilde{M}_-^i$ cover all of $M_-$ \emph{as a set}.

Now, the maps $P_- \to M_-$ and $\tilde{M}_-^i \to M_-$ give us a surjective map $P_- \cup \bigcup_i \tilde{M}_-^i \to M_-$. Taking into account the intersections, we get an induced map $P_- \cup_{Z_i} \bigcup_i \tilde{M}_-^i \to M_-$ from the gluing of $P_-$ and $\bigcup_i \tilde{M}_-^i$ along $Z=\bigcup_i Z_i$. Note however that it is not clear at all whether this glued scheme is isomorphic to $M_-$. 

It turns out that the only remaining obstruction is that $M_-$ might be non-reduced. This is a consequence of the following result, that can be seen as a strengthening of Lemma \ref{lemma:embbundle_closed}.

\begin{lemma}[cf. \cite{TX22}*{7.11}] \label{lemma:critiso_main}
Let $X \to Y$ be a proper morphism between finite type $\C$-schemes. Assume that $Y$ is reduced, and that the map $X \to Y$ induces a bijection on closed points, and isomorphisms on tangent vectors. Then $X \to Y$ is an isomorphism.
\end{lemma}

We will discuss in the next section how to determine the local structure of $M_-$ using deformation theory.

\section{Local structure via DGLA} \label{sec:localstr}

So far, we have only focused on the global structure of moduli spaces: how to determine their irreducible components, and how to describe which points lie on multiple components. We still have to describe the local structure of these spaces. To do so, we will briefly review the language of differential graded Lie algebras (or DGLAs for short), following \cite{Man09}. In our case, the local structure of the moduli space at an object $E \in \DC^b(X)$ is governed by $R\Hom(E, E)$ with the commutator induced by composition.

Now, there is a catch: computing the DGLA structure $R\Hom(E, E)$ is not an easy task. In the literature this is usually handled by replacing $E$ with a complex of injective objects (e.g. \cite{LS06}*{Appendix A}), or via Dolbeault resolutions (e.g \cite{CPZ24}*{\textsection 2.2}). For an explicit complex $E \in \DC^b(X)$, however, computing either model is not easy.

The alternative is to use a \v{C}ech cover. Informally, one wants to use the identity
\[ R\Hom(E, E) = R\Gamma(X, R\HOM(E, E)) = \check{C}(\Ucover, R\HOM(E, E)), \]
where $\Ucover$ is an affine open cover. The derived sheaf Hom $R\HOM(E, E)$ can be computed with a locally free resolution, which gives hopes to get an explicit DGLA model. However, it is not immediately clear how to produce a DGLA structure on the total complex $\check{C}(\Ucover, \cdot)$. Instead, we work with the \emph{semicosimplicial DGLA}
\[ \begin{tikzcd}[column sep=small]
\bigoplus_i \Gamma(U_i, R\HOM(E, E)) \arrow[r, shift left] \arrow[r, shift right] & \bigoplus_{i<j} \Gamma(U_{ij}, R\HOM(E, E)) \arrow[r] \arrow[r, shift left=1.5] \arrow[r, shift right=1.5] & \cdots
\end{tikzcd} \]
associated to the \v{C}ech cover. 

We will review the basics on DGLAs in Subsection \ref{subsec:DGLA}, and of semicosimplicial DGLAs in Subsection \ref{subsec:semiDGLA}. This will allow us to get an explicit functor of Artin rings controlling the deformation theory of an object $E \in \DC^b(X)$, endowed with an explicit obstruction theory.

We will then review the construction of the hull of a functor of Artin rings in Subsection \ref{subsec:versal}, highlighting how the obstruction theory allows us to describe the hull. At last, we will review the relation between the local structure on the good moduli space and the deformation functor of $E \in \DC^b(X)$ in Subsection \ref{subsec:DGLAmoduli}.

\subsection{Differential graded Lie algebras} \label{subsec:DGLA}

We follow \cite{Man09}*{\textsection 1, 4}. Recall that a \emph{differential graded Lie algebra} $L$ is the data of a complex of $\C$-vector spaces $(L^\bullet, d)$, and a graded bilinear map $[-, -]\colon L \times L \to L$ satisfying the following properties:
\begin{enumerate}
\item $[a, b]=-(-1)^{\deg a \deg b}[b, a]$,
\item $[a, [b, c]] = [[a, b],c]+(-1)^{\deg a \deg b} [b, [a, c]]$,
\item $d[a, b] = [da, b] + (-1)^{\deg a} [a, db]$. 
\end{enumerate}

Given a DGLA $L$, we associate two functors of Artin rings associated to it, as follows. First, we have the \emph{Maurer--Cartan functor} 
\[ \MC_L\colon \Art \to \Set, \qquad \MC_L(A) = \{ x \in L^1 \otimes \m_A : dx + \frac{1}{2}[x,x] = 0 \}. \]
(Here, the complex $L \otimes \m_A$ of $\C$-vector spaces is endowed with the differential $d(x \otimes a) = dx \otimes a$ and the bracket $[x \otimes a, x' \otimes a'] = [x, x'] \otimes (aa')$.) Two elements $x, y \in \MC_L(A)$ are \emph{gauge equivalent} if there exists $z \in L^0 \otimes \m_A$ such that
\[ y = e^z \ast x = x + \sum_{n=0}^\infty \frac{[z, -]^n}{(n+1)!}([z, x]-dz). \]
(Note that this is a finite sum, as $\m_A$ is a nilpotent algebra.) We let $\Def_L\colon \Art \to \Set$ to be the quotient of $\MC_L$ modulo gauge equivalence. In particular, the projection defines a smooth map $\MC_L \to \Def_L$.

\begin{example} \label{ex:DGLA_tangentspaces}
Note that $\MC_L(\C) = \Def_L(\C) = \{\ast\}$ are singletons. One quickly checks that $\MC_L(\C[\epsilon]/\epsilon^2) = Z^1(L)\otimes \C\epsilon$, while $\Def_L(\C[\epsilon]/\epsilon^2) = H^1(L) \otimes \C\epsilon$.
\end{example}

The functor $\MC_L$ admits an obstruction theory with values in $H^2(L)$, defined as follows. Consider a small extension $e\colon 0 \to M \to A \to B \to 0$, where $M\m_A = 0$. Given $y \in \MC_L(B)$, we take any lift $x \in L^2 \otimes \m_B$ of $y$. One quickly verifies that $h = dx + \frac{1}{2}[x,x]$ lies in $L^2 \otimes M$, and that $dh=0$. Moreover, the class of $h$ in $H^2(L) \otimes M$ does not depend on the lift, and it vanishes if and only if there exists a lift $x \in \MC_L(A)$ of $y$. This gives us an obstruction theory $ob_e\colon \MC_L(A) \to H^2(L)\otimes M$ for $\MC_L$. This also defines an obstruction theory for $\Def_L$, cf. \cite{Man09}*{4.13}.

\begin{remark}[\cite{Man22}*{\textsection B.3}] \label{remark:DGLA_primary}
Consider the small extension
\[ 0 \to \C xy \to \frac{\C[x,y]}{(x^2,y^2)} \to \frac{\C[x,y]}{(x,y)^2} \to 0. \]
The obstruction map $\Def_L(\C[x,y]/(x,y)^2) \to H^2(L)$ is identified with the symmetric bilinear form $[-, -]\colon H^1(L) \times H^1(L) \to H^2(L)$ induced by the Lie bracket on $L$. The associated quadratic form is the \emph{primary obstruction}, denoted by $\kappa_2$.
\end{remark}

The main source of examples for us is given by Hom complexes. Let $X$ be a smooth quasi-projective variety over $\C$. Let $E, F, G \in \DC^b(X)$ be three objects represented by finite complexes of locally free sheaves. To start, recall the \emph{Hom complex} $\HOM(E, F)^n = \bigoplus_{t-s=n} \HOM(E^s, F^t)$, with differential $d(f) = d_F^t \circ f - (-1)^n f \circ d_E^{s-1}$ for $f \in \HOM(E^s, F^t)$. There is a composition map 
\begin{equation} \label{eq:DGLAHom_composition}
\HOM(F, G) \otimes \HOM(E, F) \to \HOM(E, G)
\end{equation}
given by pointwise composition (without additional signs). This defines a morphism of complexes satisfying all the usual properties, cf. \cite{Stacks}*{Tag 0A8H}. In particular, this endows $\HOM(E, E)$ with a differential graded algebra structure.

\begin{claim}
This construction computes (a representative of) $R\HOM(E, E)$ with its differential graded algebra structure. This follows immediately from the fact that $E$ is a complex of locally free sheaves.
\end{claim}

In particular, we get that the global sections complex $\Gamma(X, \HOM(E, E))$ carries a differential graded algebra structure, and hence a DGLA by taking the commutator. We point out that for $X$ affine this computes $R\Hom(E, E)$, endowing it with a DGLA structure.

\subsection{Semicosimplicial DGLA} \label{subsec:semiDGLA}

\begin{defin}[cf. \cite{Iac10}*{p. 94}]
A \emph{semicosimplicial DGLA} $L^\Delta$ is the data of (i) a collection of DGLA $L_i$ for $i \geq 0$, and (ii) morphisms $\partial_k\colon L_{i-1} \to L_i$ for $k =0, \dots, i$, satisfying the compatibility condition $\partial_\ell \partial_k = \partial_{k+1} \partial_\ell$ for any $\ell \leq k$. 
\end{defin}

Given a semicosimplicial DGLA $L^\Delta$, we can assemble a double complex using the differentials $\partial = \partial_0 - \partial_1 + \dots$ and $d$. The associated total complex is known as the \emph{totalization} $C(L^\Delta)$. In general however, the totalization $L^\Delta$ complex does not carry a natural Lie bracket. Instead, the \emph{Thom--Whitney complex} $\Tot_{TW}(L^\Delta)$ (which we will define momentarily) is a quasi-isomorphic replacement carrying a DGLA structure. We follow \cite{Man22}*{\textsection 7}.

Denote by $\Omega_n = \C[t_0, \dots, t_n, dt_0, \dots, dt_n]/(1-\sum_i t_i, \sum dt_i)$ the dg algebra of differential forms on the $n$-simplex. For each $0 \leq k \leq n$, the face maps $\delta_k \colon (t_0, \dots, t_{n-1}) \to (t_0, \dots, t_{k-1}, 0, t_k \dots, t_{n-1})$ induce maps $\delta_k^\ast\colon \Omega_n \to \Omega_{n-1}$.

\begin{defin}[\cite{Man22}*{7.4.4}]
Let $L^\Delta$ be a semicosimplicial DGLA. The \emph{Thom--Whitney totalization} $\Tot_{TW}(L^\Delta)$ is the sub-DGLA of $\prod_n \Omega_n \otimes L_n$ whose $k$th piece is given by
\[ \left\{ (x_n) \in \prod_{n \geq 0} \bigoplus_{p+q=k} \Omega_n^p \otimes L_n^q : (\delta_i^\ast \otimes \id)x_n = (\id \otimes \partial_i) x_{n-1}, \forall \ 0 \leq i \leq n \right\}. \]
\end{defin}

\begin{lemma}[\cite{Man22}*{Theorem 7.4.5, 7}]
\begin{enumerate}
\item The integration maps $\Omega_n^n \to \C$ induce a quasi-isomorphism $\Tot_{TW}(L^\Delta) \to C(L^\Delta)$.
\item The assignment $L^\Delta \mapsto \Tot_{TW}(L^\Delta)$ is functorial. 
\item Let $f\colon L^\Delta \to M^\Delta$, $g\colon M^\Delta \to N^\Delta$ be two maps of semicosimplicial DGLA. Assume that for any $n \geq 0$, the maps $f$ and $g$ fit into a short exact sequence
\[ 0 \to L_n \to M_n \to N_n \to 0. \]
Then, we have a short exact sequence $0 \to \Tot_{TW}(L^\Delta) \to \Tot_{TW}(M^\Delta) \to \Tot_{TW}(N^\Delta) \to 0$. 
\end{enumerate}
\end{lemma}

Given a semicosimplicial DGLA $L^\Delta$, we denote by $\Def_{L^\Delta}$ the deformation functor associated to DGLA $\Tot_{TW}(L^\Delta)$.

\begin{remark}
Note that in the literature one can find a deformation functor directly associated to $L^\Delta$, cf. \cite{Iac10}*{\textsection 2.1}. If all the cohomology groups $H^i(L_k)$ vanish for $k<0$, these agree with the deformation functors of $\Tot_{TW}(L^\Delta)$, e.g. by \cite{Iac10}*{2.12}. This will not be the case in our situation, so extra care must be taken. 
\end{remark}

The main example to keep in mind is the following. Let $X$ be a smooth, projective variety, and let $E \in \DC^b(X)$ be a universally gluable object\footnote{This is automatic if $E$ lies in the heart of a bounded t-structure on $\DC^b(X)$, cf. \cite{Lie06}*{Proposition 2.1.9}.}. Up to replacing $E$ with a quasi-isomorphic complex, we may assume that $E$ is given by a finite complex of locally free sheaves. Moreover, we can arrange it so that this complex is $\Aut(E)$-equivariant, e.g. by \cite{CPZ24}*{Lemma 2.2}. 

Pick $\Ucover=\{U_1, \dots, U_N\}$ a finite affine open cover such that each $E^i$ is trivial on $U_j$. For each set of indices $I=\{i_0< \dots < i_p\}$, we have that $\Gamma(U_I, \HOM(E, E))$ is a DGLA. Set $L_p^n = \bigoplus_{i_0 < \dots < i_p} \bigoplus_{t-s=n} \Gamma(U_{i_0 \dots i_p}, \HOM(E^s, E^t))$. Given an element $f \in L^n_p$, we denote its components by $f_{i_0, \dots, i_p}^{s, t}$. We take the differential $d^i(f) = d^t \circ f -(-1)^n f \circ d^{s-1}$, and Lie bracket
\begin{equation} \label{eq:semiDGLA_commutator}
[f, g]_{i_0, \dots, i_p}^{s, s+m+n} = f_{i_0 \dots i_p}^{s+n, s+m+n} \circ g_{i_0 \dots i_p}^{s, s+n} - (-1)^{mn}g_{i_0 \dots i_p}^{s+m, s+m+n} \circ f_{i_0 \dots i_p}^{s, s+m}
\end{equation}
for $f \in L_p^m$ and $g \in L_p^n$. We assemble a semicosimplicial DGLA $L^\Delta$ with the maps $\partial_k(f)_{i_0 \dots i_n}^{s, t} = f_{i_0 \dots \hat{i}_k \dots i_n}^{s, t}$.

\begin{lemma} \label{lemma:semiDGLA_defiso}
Under the previous assumptions, we have an $\Aut(E)$-equivariant isomorphism $\Def_E \cong \Def_{L^\Delta}$. 
\end{lemma}

\begin{proof}
We will use extensively the discussion from \cite{CPZ24}*{pp. 767--8}. We consider the Dolbeault DGLA $M^n = \bigoplus_{p+q=n}\bigoplus_{t-s=p} A^{0,q}(\HOM(E^r, E^s))$. By \cite{CPZ24}*{Lemma 2.4}, the deformation functor associated to $M$ is isomorphic to $\Def_E$, and this isomorphism is $\Aut(E)$-equivariant. Thus, it suffices to construct an isomorphism between $\Def_M$ and $\Def_{L^\Delta}$. 

To do so, note that there is a natural map $L^\Delta \to M$ of semicosimplicial DGLA, where we identify $M$ with a semicosimplicial DGLA in degree zero. The induced map at the level of total complexes is a quasi-isomorphism, and thus the map on the Thom--Whitney complexes is a quasi-isomorphism as well. The equivalence between the deformation functors is now direct, cf. \cite{Man22}*{Theorem 6.6.2}.
\end{proof}

\subsection{Finding hulls} \label{subsec:versal}

Let $L$ be a DGLA over $\C$ with $H^k(L_i)=0$ for $k<0$, and assume that $H^1(L)$ is finite dimensional. By our discussion in Subsection \ref{subsec:DGLA}, it turns out that the tangent space of $\Def_L$ is finite dimensional as well. In other words, the functor $\Def_L$ satisfies Schlessinger's (H3) condition (see \cite{Har10}*{p. 113}). One verifies that the conditions (H0)--(H2) hold as well, and so $\Def_L$ admits a hull by Schlessinger's criterion. This is, there exists a complete local Noetherian $\C$-algebra $R$ and an étale map $h_R \to \Def_L$.

The construction of $R$ is not canonical, as it depends on various choices. Let us briefly recall how this can be handled via the obstruction map (cf. \citelist{\cite{Art76}*{\textsection 7} \cite{Har10}*{\textsection 16}}). We let $D\colon \Art \to \Set$ be a functor of Artin rings satisfying conditions (H0)--(H3), endowed with an obstruction theory with values in $V$. 

First, let $s_1, \dots, s_r$ be a basis of the dual of $D(\C[\epsilon]/\epsilon^2)$. We set $S=\C[[s_1, \dots, s_r]]$ and $\m_S = (s_1, \dots, s_r)$ its maximal ideal. We will inductively construct ideals $J_q \subseteq \m_S^{q+1}$ and elements $\xi_q \in D(S/J_q)$ forming a formal family. These will give us a map $h_R \to D$, which will turn out to be étale.

To do so, start by taking $J_1 = \m_S^2$ and $\xi_1$ the canonical element. Inductively, assume that we have constructed $(J_q, \xi_q)$, and consider the small extension $0 \to J_q/\m J_q \to S/\m J_q \to S/J_q \to 0$. The obstruction map gives us a class $ob(\xi_q) \in V \otimes_\C J_q/\m J_q$. Let $J_q/J_{q+1}$ be the largest quotient of $J_q/\m J_q$ such that the image of $ob(\xi_q)$ vanishes. This defines an ideal $\m J_q \subseteq J_{q+1} \subseteq J_q$. We get the diagram:
\[ \begin{tikzcd} 0 \arrow[r] & J_q/\m_S J_q \arrow[r] \arrow[d] & S/\m J_q \arrow[r] \arrow[d] & S/J_q \arrow[r] \arrow[d] & 0 \\ 0 \arrow[r] & J_q/J_{q+1} \arrow[r] & S/J_{q+1} \arrow[r] & S/J_q \arrow[r] & 0. \end{tikzcd} \]
The vanishing of the obstruction guarantees that $\xi_q$ lifts to $D(S/J_{q+1})$, and that $J_{q+1}$ is minimal with this property among ideals between $\m J_q$ and $J_q$. 

We continue inductively, and let $J = \bigcap_q J_q$. The quotient $R=S/J = \varprojlim S/J_q$ is the required hull; the elements $\xi_q$ define a map $h_R \to D$. By construction, we get that $J_q = J + \m_S^{q+1}$.

\begin{prop} \label{prop:versal_stopcriterion}
Let $A = \C[[s_1, \dots, s_r]]/I$ be a complete local Noetherian $\C$-algebra with $I \subset \m_S^2$. Suppose that there exists a morphism $h_A \to D$ inducing an isomorphism in tangent spaces, and let $\phi\colon R \to A$ be a lift to $A$.
\begin{enumerate}
\item The map $\phi$ is of the form $\phi=\phi_1 + \phi_2 + \dots$, where $\phi_i$ is a sum of monomials of degree $i$, and $\phi_1$ is a linear isomorphism.
\item Assume that there exists $d \geq 2$ such that $\m_S^d \cap I \subseteq \m_S I$, and such that the natural map $J+\m_S^d \to I+\m_S^d$ is an isomorphism. Then we get that $I \to J$ is an isomorphism; in particular, $R\cong A$. 
\end{enumerate}
\end{prop}

This lemma is already interesting for $d=2$, for which $I=0$. Informally, the idea is that the hull is a quotient of $\C[[s_1, \dots, s_r]]$. But $A$ already is a ``maximal family'', so there is nothing else to be checked. A version of this result for $d=3$ was implicitly used in \citelist{\cite{Xia18}*{Corollary 4.22} \cite{TX22}*{pp. 405--6}}.

In practice, Proposition \ref{prop:versal_stopcriterion} will give us a way to verify whether a family $h_A \to F$ is versal. We first represent $A= \C[[s_1, \dots, s_r]]/(f_1, \dots, f_\ell)$, and pick some $d$ such that $\m_S^d + I \subseteq \m I$. Then, the remaining verification is done by computing obstructions up to this degree. The key point is that we have a bound on when to stop computing these obstructions.

\begin{proof}
The first part is clear as both $h_R \to D$ and $h_A \to D$ are isomorphisms in tangent spaces. For the second one, we have that $\phi$ lifts to an isomorphism $\psi\colon \C[[s_1, \dots, s_r]] \to \C[[s_1, \dots, s_r]]$, fitting into the diagram:
\[ \begin{tikzcd} J \arrow[r, hook] \arrow[d] & {\C[[s_1, \dots, s_r]]} \arrow[r] \arrow[d, "\psi"] & R \arrow[d, "\phi"] \\ I \arrow[r, hook] & {\C[[s_1, \dots, s_r]]} \arrow[r] & A. \end{tikzcd} \]
Note that $R \to A$ is surjective, and so $J \to I$ is injective. By assumption, we have that $J+\m_S^d = I + \m_S^d$ is an isomorphism, where we identify $J$ as a submodule of $I$. This way, we can find $g_i \in J$ such that $f_i -g_i \in \m_S^d\cap I$. In other words, we have shown that $I=J + \m_S^d \cap I$. At last, by assumption this yields that $I=J+\m I$. By Nakayama this gives us $I=J$, as claimed. 
\end{proof}

\begin{cor} \label{cor:versal_quadratic}
In the notation of Proposition \ref{prop:versal_stopcriterion}, assume that $I\cap \m_S^3 \subseteq \m_S I$, and that $I+\m_S^3$ equals $\m_S^3 + \ker \kappa_2$, where $\kappa_2$ is the primary obstruction (cf. Remark \ref{remark:DGLA_primary}). Then $R \cong T$ holds.
\end{cor}

\begin{proof}
The ideal $J_2$ is given by the kernel of the primary obstruction, cf. \cite{Man09}*{\textsection 5}.
\end{proof}

\begin{remark}
Note that in this discussion we have not kept track of the $\Aut(E)$-action. We refer to \cite{CPZ24}*{\textsection 2.3} for a discussion on how to do so by using the Kuranishi map. 
\end{remark}

\subsection{From DGLA to the moduli space} \label{subsec:DGLAmoduli}

The goal of this subsection is to relate the DGLA formalism to the local structure of the moduli space, following the discussion on \cite{CPZ24}*{\textsection 2.4} (see also \cite{AS25}*{\textsection 2.3}). To do so, let $X$ be a smooth, projective variety and let $\sigma = (Z, \A)$ be a Bridgeland stability condition on $X$. In particular, for each vector $v$ we have good moduli spaces $M_\sigma(v)$.

Let us parse this out. Given a numerical vector $v \in K^{num}(X)$, there is an open substack $\M_\sigma(v) \subset \M_{pug}(X)$ whose $S$-valued points are perfect, universally gluable complexes $E$ such that $E|_s$ is a $\sigma$-semistable for all $s \in S$. By Theorem \ref{teo:existsmoduli_main}, this stack admits a good moduli space $M_\sigma(v)$, whose closed points parametrize S-equivalence classes of $\sigma$-semistable objects with numerical vector $v$.  

This way, given a closed point $p \in M$, we let $[E] \in \M(\C)$ be the unique closed point over $p$, so that $E$ is polystable. Consider the deformation functor $\Def_E$, which admits a $\Aut(E)$-equivariant hull of the form $\C[[\Ext^1(E, E)]]/I$.

\begin{lemma}[\cite{CPZ24}*{Lemma 2.7}] \label{lemma:DGLAmoduli_gms}
Assume that $\Aut(E)$ is linearly reductive. Then we have $\hat{\O}_{M, p} \cong (\C[[\Ext^1(E, E)]]/I)^{\Aut(E)}$. 
\end{lemma}

Note that when $E$ is stable, the action of $\Aut(E)=\C$ is trivial, and so we get an isomorphism with $\C[[\Ext^1(E, E)]]/I$ directly (see also the discussion in \cite{Lie06}*{\textsection 4.3}). In general however, we do need to keep track of the $\Aut(E)$-action.

\section{A single curve} \label{sec:single}

In this section we will apply the technical machinery we have constructed to analyze the stability conditions arising from the contraction to a single rational curve. This has been extensively studied in \cite{Tod13} for the contraction of a $(-1)$-curve and in \cite{TX22} for the contraction of a $(-n)$-curve.

We consider the following setup. Let $S$ be a smooth, projective surface, and let $f\colon S \to T$ be a birational morphism to a normal, projective surface. We assume that $\Exc(f)=C$, where $C$ is a smooth, rational curve with self-intersection $-n$. Fix $\eta \in \Amp(T)_\Q$ and $\beta \in \NS(S)_\Q$ such that $-1<\beta.C-n/2<0$. We let $\overline{\sigma}_{\beta, f^\ast\eta}$ be the stability condition with central charge $Z_{\beta, f^\ast\eta}$, whose existence is guaranteed by Theorem \ref{teo:limits_all}. 

Using Bridgeland's deformation theorem, there is a continuous family $\sigma_\epsilon\in \Stab(S)$ of stability conditions with central charge $Z_{\beta, f^\ast\eta+\epsilon C}$, provided that $\abs{\epsilon} \ll 1$, with $\sigma_0 = \overline{\sigma}_{\beta, f^\ast\eta}$. We denote by $M_\epsilon = M_{\sigma_\epsilon}([pt])$.

Note that for small $\epsilon <0$, we have that $f^\ast\eta+\epsilon C$ is ample in $S$. This way, for small $\epsilon<0$, the stability condition $\sigma_{\epsilon}$ agrees with the Arcara--Bertram stability conditions $\sigma_{\beta, f^\ast\eta+\epsilon C}$. In particular, this gives us that $M_\epsilon \cong S$ in this case, with a universal family given by the structure sheaf of the diagonal.

\begin{prop} \label{prop:single_main}
Under the previous assumptions, we have that the $\sigma_0$-polystable objects of phase 1 and numerical vector $[pt]$ are $\O_x$ for $x \notin C$, and $\O_C \oplus \O_C(-1)[1]$. The good moduli space $M_0$ is isomorphic to $T$. 

Lastly, for $\epsilon >0$ we have that $M_\epsilon$ only parametrizes stable objects. It is isomorphic to $T$ if $n=1$, to $S$ if $n=2$, and to $S \sqcup_C \P^{n-1}$ if $n \geq 3$, where $C \subset \P^{n-1}$ is embedded as a rational normal curve. 
\end{prop}

The proof of this proposition will be the focus of the rest of this section. Compare this result to \cite{Tod13}*{Theorem 3.16} and \cite{TX22}*{\textsection 7--8}, where most of the proposition was already proven. The only part that is new is the fact that $M_0 \cong T$, which involves describing the local structure of $M_0$ at the point corresponding to the strictly semistable object.

We will divide the proof of Proposition \ref{prop:single_main} into three steps. First, we will determine the semistable objects in Subsection \ref{subsec:singless}. We will then use this to determine the irreducible components of the corresponding moduli spaces in Subsection \ref{subsec:singleirr}. Finally, we will describe the local structure in Subsection \ref{subsec:singlelocal}.

\subsection{Semistable objects} \label{subsec:singless}

To determine the semistable objects for the stability conditions $\sigma_\epsilon$ we start by looking at $\sigma_0$. To do so, we let $E$ be an object in the heart of $\sigma_0$, of phase 1 and numerical class $[pt]$. By \cite{Vil25}*{Proposition 3.13}, we have that $\sigma_0$ is an extension of objects supported on $\Exc(f)$ union finitely many points. Moreover, we have that the only objects supported on $\Exc(f)$ that are stable are $\O_C$ and $\O_C(-1)[1]$, thanks to \cite{Vil25}*{Lemma 5.8}. 

This way, we have two types of polystable objects with numerical class $[pt]$: $\O_x$ for $x \in S \setminus C$, and $\O_C \oplus \O_C(-1)[1]$. We point out that for $x \in C$, we have the short exact sequence
\begin{equation} \label{eq:singless_point}
0 \to \O_C \xrightarrow{\alpha} \O_x \xrightarrow{\beta} \O_C(-1)[1] \to 0.
\end{equation}

This gives us a complete description of the polystable objects parametrized by $M_0$. Note that only one of them is strictly semistable, namely, $\O_C \oplus \O_C(-1)[1]$. With this in mind, we can compute the walls corresponding to the collection of $\sigma_0$-polystable objects, by using Theorem \ref{teo:wc_main} (and especially \eqref{eq:wc_wallformula}).

We get a single numerical wall, given by the preimage of 
\begin{equation} \label{eq:singless_wall}
\{Z \in \Hom(\Lambda, \C) : \Re Z(\O_C) \cdot \Im Z([pt]) = \Re Z([pt]) \cdot \Im Z(\O_C) \}
\end{equation}
via the projection map $\Stab(S) \to \Hom(\Lambda, \C)$.

\begin{remark}
Note that the set from \eqref{eq:singless_wall} describes a (real) submanifold of codimension 1 inside of $\Hom(\Lambda, \C)$. One quickly verifies that the path $\{\sigma_\epsilon\}$ is transversal to the preimage of this submanifold.
\end{remark}

Using the description of Theorem \ref{teo:wc_main}, we get that for $\epsilon<0$, the $\sigma_\epsilon$-stable objects are $\{\O_x : x \in S \setminus C\}$, together with objects $E$ fitting in a triangle $\O_C \to E \to \O_C(-1)[1] \xrightarrow{\xi} \O_C[1]$ for some $\xi \neq 0$. Here $E \cong \O_x$ for some $x \in C$. As we mentioned above, the stability condition $\sigma_\epsilon$ agrees with the Arcara--Bertram stability condition $\sigma_{\beta, f^\ast\eta+\epsilon C}$. We recover the description of Proposition \ref{prop:AB_moduli} at the level of closed points.

On the other hand, for $\epsilon>0$, the $\sigma_\epsilon$-stable objects are $\{\O_x : x \in S \setminus C\}$, together with objects $E$ fitting in a triangle $\O_C(-1)[1] \to E \to \O_C \xrightarrow{\xi} \O_C(-1)[2]$ for some $\xi \neq 0$. Note here that $\Hom(\O_C, \O_C(-1)[2]) = \Ext^2(\O_C, \O_C(-1))$ is $n$-dimensional, and thus the nontrivial extensions (up to scalar) are parametrized by a projective space of dimension $n-1$.

From this discussion we get that there are three distinct moduli spaces, depending on whether $\epsilon$ is positive, zero, or negative. We will denote them by $M_+$, $M_0$, and $M_-\cong S$ respectively.

\subsection{Irreducible components} \label{subsec:singleirr}

Our next goal is to describe the irreducible components of $M_+$, using the results from Section \ref{sec:irrviaWC}. Our starting point is the fact that\footnote{Note that there is a small difference in notation: in Section \ref{sec:irrviaWC}, we denoted by $M_+$ the moduli space that was known to us, and by $M_-$ the one obtained after wall-crossing. Here $M_-$ is the moduli space that we know.} $M_- \cong S$, with universal family $\U_{M_-} \cong \O_{\Delta}$, and that there is a single wall.

Given $x \in S$, note that $\O_x$ admits a non-zero map from $\O_C$ if and only if $x \in \O_C$. To show this, note that $\O_C$ and $\O_x$ have disjoint supports if $x \notin C$. Thus, the locus $P \subset M_-$ of destabilized objects corresponds exactly to the curve $C$. 

Following Subsection \ref{subsec:em}, we blow-up the locus $P \subset M_-$ and construct an elementary modification of the family $\U_{M_-}$. In our case, the blow-up does not change the scheme $S$ (as $C \subset M_-$ is already an effective Cartier divisor). The new objects corresponding to the points of $C \subset S$ correspond to extensions parametrized by $\Ext^2(\O_C, \O_C(-1))$. Let us identify precisely which extension classes arise from this elementary modification. x

Given $x \in C$, let $\alpha$ and $\beta$ be the maps given by the triangle:
\begin{equation} \label{eq:singleirr_point}
\O_C \xrightarrow{\alpha} \O_x \xrightarrow{\beta} \O_C(-1)[1] \to \O_C[1].
\end{equation}
In this notation, we have that \eqref{eq:em_tangentmap} gives us the diagram
\begin{equation} \label{eq:singleirr_preem}
\begin{tikzcd}
0 \arrow[r] & T_{C, x} \arrow[r] \arrow[d] & T_{S, x} \arrow[r] \arrow[d, "KS", "\cong"'] & N_{C/S, x} \arrow[r] \arrow[d] & 0 \\ 
0 \arrow[r] & K_x \arrow[r] & \Ext^1(\O_x, \O_x) \arrow[r, "\beta \circ - \circ \alpha"'] & \Ext^2(\O_C, \O_C(-1)) 
\end{tikzcd}
\end{equation}
By chasing the diagram, one checks that the bottom right map has rank one. It follows that $K_x$ is one-dimensional, and so the map $N_{C/S, x} \to \Ext^2(\O_C, \O_C(-1))$ is injective. 

From Corollary \ref{cor:em_newfamily}, we get an induced map $S \to M_+$. From the discussion of Subsection \ref{subsec:singless} we get that this is injective on closed points, and a direct computation shows that it is also injective on tangent vectors.

On the other hand, note that there is a second way to produce objects on $M_+$, by taking $Q = \P \Ext^2(\O_C, \O_C(-1))$. This carries a universal extension of the form
\[ \O_C(-1)[1] \boxtimes \O(1) \to \U \to \O_C \to \ast \]
inside $\DC^b(Q \times S)$. This gives us a second map $Q \to M_+$, injective on closed points and tangent vectors. 

Moreover, we have that the union of the images of $S$ and $Q$ inside $M_+$ contains all the closed points of $M_+$. This way, it remains to determine whether the maps $S, Q \to M_+$ are closed embeddings, and how these two subschemes are glued. 

To do so, let us look more carefully at the extensions parametrized by $Q$. Given a non-zero $\xi \in \Ext^2(\O_C, \O_C(-1))$, we consider the extension
\begin{equation} \label{eq:singleirr_newext}
\O_C(-1)[1] \xrightarrow{\gamma} E \xrightarrow{\delta} \O_C \xrightarrow{\xi} \ast
\end{equation}
induced by it. Here $E$ represents a closed point $p$ of $M_+$. To compute the local structure of $M_+$ around this point, we will use that the tangent space of $M_+$ at $E$ is identified with $\Ext^1(E, E)$. We compute this group using the maps from \eqref{eq:singleirr_newext}.

\begin{claim} \label{claim:singleirr_Ext1}
Given $\xi$ as before, we have an exact sequence
\[ 0 \to \C \xi \hookrightarrow \Ext^2(\O_C, \O_C(-1)) \xrightarrow{\delta \circ - \circ \gamma} \Ext^1(E, E) \xrightarrow{\gamma \circ - \circ \delta} K \to 0, \]
where $K = \ker \left( \xi \circ -\colon \Ext^0(\O_C(-1), \O_C) \to \Ext^2(\O_C(-1). \O_C(-1)) \right)$.
\end{claim}

This is still non-satisfactory, as we need to determine $K$. To do so, we will explicitly compute the map $\xi \circ -$ by carefully identifying the map that defines $K$.

First, fix a basis $e_0, e_1$ of $\Hom(\O_C, \O_C(1))$. This induces bases $\{e_0^i e_1^j : i+j=k \}$ of $\Hom(\O_C(a), \O_C(a+k))$ for $k \geq 0$ and $a \in \Z$. By Serre duality, we get that
\[ \Ext^2(\O_C, \O_C(-1)) \cong \Hom(\O_C(-1), \O_C(n-2))^\vee, \]
and so it carries a dual basis $\{ (e_0^i e_1^j)^\vee : i+j=n-1 \}$.

\begin{lemma} \label{lemma:singleirr_jumplocus}
Let $\xi = \sum_{i+j=n-1} a_{i,j} (e_0^i e_1^j)^\vee$ be an element of $\Ext^2(\O_C, \O_C(-1))$. We have that the map
\begin{equation} \label{eq:singleirr_compxi}
\xi \circ -\colon \Ext^0(\O_C(-1), \O_C) \to \Ext^2(\O_C(-1). \O_C(-1))
\end{equation}
has rank $0$ (resp. $\leq 1$) if and only if $\xi =0$ (resp. $a_{i,j} = b_0^ib_1^j$ for some $b_0, b_1$). 
\end{lemma}

\begin{proof}
Given $e \in \Hom(\O_C(-1), \O_C)$, consider the map $-\circ e \colon \Ext^2(\O_C, \O_C(-1)) \to \Ext^2(\O_C(-1), \O_C(-1))$. By Serre duality we have the diagram
\[ \begin{tikzcd}[column sep = large]
\Ext^2(\O_C, \O_C(-1)) \arrow[r, "-\circ e"] \arrow[d, "\cong", "SD"', no head] & \Ext^2(\O_C(-1), \O_C(-1)) \arrow[d, "\cong", "SD"', no head] \\ \Hom(\O_C(-1), \O_C(n-2))^\vee \arrow[r, "(e \circ -)^\vee"'] & \Hom(\O_C(-1), \O_C(n-3))^\vee. 
\end{tikzcd} \]
This way, to compute $- \circ e$, we instead compute the map
\[ e \circ -\colon \Hom(\O_C(-1), \O_C(n-3)) \to \Hom(\O_C(-1), \O_C(n-2)) \]
in the bases from before, and take its transpose. 

This allows us to compute the map \eqref{eq:singleirr_compxi}: on the given bases, it is given by
\[ \begin{pmatrix} a_{0, n_1} & a_{1, n_2} & \cdots & a_{n-2, 1} \\ a_{1, n-2} & a_{2, n-3} & \cdots & a_{n-1,0} \end{pmatrix}^t. \]
From here determining the rank in terms of the $a_{i,j}$ is immediate.
\end{proof}

\begin{cor}
Let $U \subset Q$ be the locus of objects whose corresponding extension class $\xi \in \Ext^2(\O_C, \O_C(-1))$ is not of the form $\xi = \sum_{i+j=n-1} b_0^ib_1^j (e_0^ie_1^j)^\vee$. Then the map $U \to M_+$ is an open embedding. 
\end{cor}

\begin{proof}
Fix $p \in U$, and let $E$ be the corresponding class. From Claim \ref{claim:singleirr_Ext1} we have that $\Ext^1(E, E)$ has dimension $n-1$. In particular, the map $U \to M_-$ is injective on closed points and induces an isomorphism on tangent spaces. As $\O_{U, p}$ is regular, this follows immediately by Proposition \ref{prop:versal_stopcriterion} (and the discussion thereafter).
\end{proof}

The locus $Z = Q \setminus U$ is a bit more delicate to describe. Fix $\C \xi \in Z$; from Lemma \ref{lemma:singleirr_jumplocus}, we have that $\xi=\sum_{i+j=n_1}b_0^i b_1^j (e_0^ie_1^j)^\vee$.

\begin{claim} \label{claim:singleirr_gluedlocus}
Note that $\xi$ is an extension class arising from the elementary modification performed on $S$. To show this, we let $\phi = b_1 e_0-b_0 e_1 \in \Ext^0(\O_C(-1), \O_C)$, which satisfies $\phi \circ \xi =0$. The map $\phi$ defines a point $x \in S$, seen as the extension class in \eqref{eq:singleirr_point}. By chasing out the diagram \eqref{eq:singleirr_preem}, we claim that the image of the leftmost map is $\C \xi$.

In fact, note that the image of $\Ext^1(\O_x, \O_x) \to \Ext^2(\O_C, \O_C(-1))$ equals the kernel of $\phi \circ -\colon \Ext^2(\O_C, \O_C(-1)) \to \Ext^2(\O_C, \O_C)$. The same strategy of Lemma \ref{lemma:singleirr_jumplocus} allows us to describe this map with the matrix
\[ \begin{pmatrix} b_1 & -b_0 \\ & b_1 & -b_0 \\ & & \smallddots & \smallddots \\ & & & b_1 & -b_0 \end{pmatrix}. \]
The kernel is now clearly spanned by $\xi$. 
\end{claim}

This way, the image of $Z \subset Q$ corresponds to the image of $C \subset S$ inside of $M_+$. In other words, the two maps $S \to M_+$, $Q \to M_+$ induce a map $S \sqcup_C Q \to M_+$, where $S \sqcup_C Q$ is the gluing of $S$ and $Q$ along $C$. This is proper and bijective on closed points. One easily verifies that the map is also injective on tangent vectors. Note however that we cannot use Lemma \ref{lemma:critiso_main} directly, as we do not know whether $M_+$ is reduced.

\subsection{Local structure} \label{subsec:singlelocal}

To finish up the proof of Proposition \ref{prop:single_main}, we need to compute the local structure of the moduli spaces $M_\epsilon$. To do this, we will use our discussion from Section \ref{sec:localstr}. 

To describe the local structure of $M_0$ and $M_+$, we need to compute the deformation functor associated to $\O_C \oplus \O_C(-1)[1]$, and to extensions of the form $\O_C(-1)[1] \to E \to \O_C \to \O_C(-1)[2]$, respectively. Note that both objects are supported on $C$. This way, we can compute the deformation functor associated to $E$ by working on a local model of $(C \subset S)$, due to the following lemma.

\begin{lemma}[cf. \cite{TX22}*{Lemma 7.4}]
Let $X$ be a finite type $\C$-scheme and let $Z \subset X$ be a closed subset. Denote by $\DC^b_Z(X)$ the full subcategory of $\DC^b(X)$ of objects supported set-theoretically at $Z$, and let $\hat{X}$ be the completion of $X$ along $Z$. Then $\DC^b_Z(X) \cong \DC^b_{Z}(\hat{X})$. 
\end{lemma}

This way, we let $X$ to be the total space of $\O_{\P^1}(-n)$. This is a toric variety, and we use \cite{CLS11}*{\textsection 10.1--2} to get explicit charts. Here, the variety $X$ is obtained by gluing $U_1 = \Spec \C[x, u]$ and $U_2 = \Spec \C[y, v]$ along $\C[x, u]_x \cong \C[y, v]_y$, identifying $y=x^{-1}, v=ux^n$. 

We let $C$ be the zero section of the bundle, described as the zero locus of $u$ in $U_1$. We point out that $(C \subset X)$ is a local model for $(C \subset S)$.

Let $L$ be the fiber of the bundle $X \to \P^1$ given by the zero locus of $x$ in $U_1$. This way, we get that $L \cong \mathbb{A}^1$ and $C \cong \P^1$. The intersection of $L$ and $C$ is transversal, and $C^2 = -n$. At last, the line bundles $\O_X(-C)$ and $\O_X(L)$ have transition functions $f_{12}=y^n$, resp. $g_{12}=y$. 

We take the resolutions $\O_C = [\O_X(-C) \to \O_X]$ and $\O_C(-1) = [\O_X(-C-L) \to \O_X(-L)]$. On the other hand, we have
\begin{align*}
\Aut(\O_C \oplus \O_C(-1)[1]) &\cong \Aut(\O_C) \times \Aut(\O_C(-1)) \cong \begin{pmatrix} \C^\ast & 0 \\ 0 & \C^\ast \end{pmatrix}
\end{align*}
This matrix group acts naturally on the resolution
\begin{equation} \label{eq:singlelocal_res}
\O_C \oplus \O_C(-1)[1] \cong [\O_X(-C-L) \to \O_X(-L) \oplus \O_X(-C) \to \O_X].
\end{equation}

\begin{lemma} \label{lemma:singlelocal_hullE}
Consider the \v{C}ech semicosimplicial DGLA $L_0 \rightrightarrows L_1$ induced by the cover $\Ucover = \{ U_1, U_2 \}$ on $\Hom(E, E)$, where $E$ is the resolution \eqref{eq:singlelocal_res}. We have that $\Aut(E) \cong (\C^\ast)^2$, and that $\Def_E$ admits an $\Aut(E)$-equivariant hull $\C[[p_1, \dots, p_n, q_0, q_1]]/(p_iq_0+p_{i+1}q_1, i=1, \dots, n-1)$, where $(\C^\ast)^2$ acts with weight $(-1, 1)$ on the $p_i$, and $(1, -1)$ on the $q_j$. 
\end{lemma}

\begin{proof}
First of all, recall that the Thom--Whitney totalization $\Tot_{TW}(L^\Delta)$ is quasi-isomorphic to $R\Hom(E, E)$. This way, we can compute the dimension of the cohomology groups of $\Tot_{TW}(L^\Delta)$ by computing $R^i\Hom(E, E)$. We get that $\dim H^i(\Tot_{TW}(L^\Delta)) =2, n+2, 2n-2, n-2$ for $i=0, 1, 2, 3$, and zero otherwise. Moreover, we have that $R^1\Hom(E, E) \cong \Ext^2(\O_C, \O_C(-1)) \oplus \Hom(\O_C(-1), \O_C)$. This helps us finding explicit representatives for a basis of $H^1(\Tot_{TW}(L^\Delta))$:
\begin{itemize}
\item Consider the map of complexes $R\Hom(\O_C, \O_C(-1)[1]) \to R\Hom(E, E)$ 
induced by the resolutions for $\O_C$ and $\O_C(-1)[1]$. Here, a basis for the $n$-dimensional space $R^1\Hom(\O_C, \O_C(-1)[1])$ can be lifted to the \v{C}ech cover, hence to elements in $\Gamma(U_{12}, \HOM(\O_X(-C), \O_X(-L)))$. This induces elements $\alpha_1, \dots, \alpha_n$ in $\Tot_{TW}(L^\Delta)^1$, given by the formulas
$(\alpha_i)_0=0$, $(\alpha_i)_1 = dt_0 \otimes \overline{\alpha}_i$, where $(\overline{\alpha}_i)_{12}^{-1,-1} = \begin{pmatrix} 0 & y^i \\ 0 & 0 \end{pmatrix}. $

\item Consider the map of complexes $R\Hom(\O_C(-1)[1], \O_C) \to R\Hom(E, E)$. constructed as above. By looking at the image on degree 1, we get a two-dimensional subspace of $R^1\Hom(E,E)$. We lift a basis to get representatives $\beta_0, \beta_1 \in \Tot_{TW}(L^\Delta)^1$ given by the formulas $(\beta_j)_0 = 1 \otimes \overline{\beta}_j$, $(\beta_j)_1 = 1 \otimes \partial_1\overline{\beta}_j$, with
\begin{gather*}
(\overline{\beta}_0)_1^{-2,-1} = \begin{pmatrix} 0 \\ 1 \end{pmatrix}, \quad (\overline{\beta}_0)_2^{-2, -1} = \begin{pmatrix} 0 \\ y \end{pmatrix}, \\ (\overline{\beta}_0)_1^{-1,0} = \begin{pmatrix} 1 & 0 \end{pmatrix}, \quad (\overline{\beta}_0)_2^{-1,0} = \begin{pmatrix} y & 0 \end{pmatrix}; \\
(\overline{\beta}_1)_1^{-2,-1} = \begin{pmatrix} 0 \\ x \end{pmatrix}, \quad (\overline{\beta}_1)_2^{-2, -1} = \begin{pmatrix} 0 \\ 1 \end{pmatrix}, \\ (\overline{\beta}_1)_1^{-1,0} = \begin{pmatrix} x & 0 \end{pmatrix}, \quad (\overline{\beta}_1)_2^{-1,0} = \begin{pmatrix} 1 & 0 \end{pmatrix}.
\end{gather*}
\end{itemize}
Note now that $[\alpha_i, \alpha_{i'}]=[\beta_j,\beta_{j'}]=0$, which can be checked directly from \eqref{eq:semiDGLA_commutator}. On the other hand, we have that $[\alpha_i, \beta_j] = \gamma_{i-j}$, with $(\gamma_k)_0=0$, $(\gamma_k)_1 = dt_0 \otimes \overline{\gamma}_k$, for $(\overline{\gamma}_k)_{12}^{-2,-1} = \begin{pmatrix} y^k \\ 0 \end{pmatrix}, (\overline{\gamma}_k)_{12}^{-1,0} = \begin{pmatrix} 0 & - y^k \end{pmatrix}$. 
Here, we have that the images of $\{\gamma_1, \dots, \gamma_{n-1} \}$ in $\Ext^2(E, E)$ are linearly independent. We also have $\gamma_0 = d \mu$ and $\gamma_n =d\eta$, where
\begin{gather*}
(\mu)_0 = 1 \otimes \overline{\mu}, \quad (\mu)_1 = t_0 \otimes \partial_1 \overline{\mu}, \qquad \overline{\mu}_1^{-2,-1} = \begin{pmatrix} 1 \\ 0 \end{pmatrix}, \quad \overline{\mu}_1^{-1,0} = \begin{pmatrix} 0 & -1 \end{pmatrix}; \\
(\eta)_0 = 1 \otimes \overline{\eta}, \quad (\eta)_1 = t_1 \otimes \partial_0 \overline{\eta}, \qquad \overline{\eta}_2^{-2,-1} = \begin{pmatrix} -1 \\ 0 \end{pmatrix}, \quad \overline{\eta}_2^{-2,-1} = \begin{pmatrix} 0 & 1 \end{pmatrix}.
\end{gather*}
It is clear that $[\alpha_i, \eta]$, $[\alpha_i, \mu]$, $[\mu, \mu]$, $[\eta,\mu]$, $[\eta, \eta]$ are all equal to zero. A fast computation guarantees that $[\beta_j, \eta]=[\beta_j, \mu]=0$ as well.

With this in mind, denote by $p_1, p_2, \dots, p_n, q_1, q_2$ the dual basis of the images of $\alpha_1, \alpha_2, \dots, \alpha_n, \beta_1, \beta_2$ in $\Ext^2(E, E)$. The automorphism group $(\C^\ast)^2$ acts with weights $(-1, 1)$ on the $p_i$, and $(1, -1)$ on the $q_j$. With the lifts chosen above, we have that the hull of $\Def_L$ is $R=\C[[p_1, \dots, p_n, q_1, q_2]]/(p_kq_0 + p_{k+1}q_1: k=1, \dots, n-1)$. By construction, this is compatible with the $(\C^\ast)^2$-action. 
\end{proof}

\begin{cor}
We have that $M_0 \cong T$. 
\end{cor}

\begin{proof}
Note that there is a natural map $S \to M_0$, mapping $x$ to the S-equivalence class of $[\O_x]$. The map only identifies the points of $C$ to a point. This way, it suffices to compute the local rings of $M_0$.

First, note that $\O_x$ is $\sigma_0$-stable if $x \notin C$. This way, the ring $\hat{\O}_{M_0, [\O_x]}$ is regular, as we can apply Lemma \ref{lemma:DGLAmoduli_gms} together with the fact that $\O_x$ has unobstructed deformations.

Second, we need to compute $\hat{\O}_{M_0, E}$. This follows directly from the previous computation and by Lemma \ref{lemma:DGLAmoduli_gms}. Here the $(\C^\ast)^2$-invariants of $\C[[p_i, q_j]]$ correspond to the subalgebra generated by the binomials $p_iq_j$. It follows that the $(\C^\ast)^2$-invariants of $R$ are
\begin{align*}
& \C[[p_iq_j: 1 \leq i \leq n, j=1,2]]/(p_iq_0 + p_{i+1}q_1: i=1, \dots, n) \\
\cong& \C[[s_0, \dots, s_n]]/\left(\rk \begin{pmatrix} s_0 & s_1 & \dots & s_{n-1} \\ s_1 & s_2 & \dots & s_n \end{pmatrix} \leq 1\right), 
\end{align*}
which is exactly the germ of an $\frac{1}{n}(1,1)$-singularity.
\end{proof}

This concludes the local structure of $M_0$. Compare to \cite{Tod13}*{Theorem 3.16(i)}, where the case $n=1$ was solved with a different method.

At last, let us discuss the case for $M_+$. We have the following statement.

\begin{prop}[\cite{TX22}*{\textsection 8}] \label{prop:singlelocal_reduced}
The moduli space $M_+$ is reduced. 
\end{prop}

Note that this was proven in \cite{TX22}*{\textsection 8} by appealing to Corollary \ref{cor:versal_quadratic}, together with a computation of the primary obstruction map. As such, we will only give a sketch of our argument; a similar computation will be performed in Subsection \ref{subsec:interlocal}.

\begin{proof}[Proof of Proposition \ref{prop:singlelocal_reduced}] (Sketch) 
We have
\[ \Ext^1(\O_C, \O_C(-1)[1]) \cong \Ext^1(\O_X(-C), \O_X(-L)). \]
This way, given an extension $\O_C(-1)[1] \to F \to \O_C \to \O_C(-1)[1]$, the corresponding class in $\Ext^1(\O_X(-C), \O_X(-L))$ allows us to construct a resolution of $F$ by locally free sheaves, replacing the middle term of \eqref{eq:singlelocal_res} by a non-trivial extension. We then compute the Thom--Whitney totalization and the hull, in the same spirit as the proof of Lemma \ref{lemma:singlelocal_hullE}.
\end{proof}

\section{Multiple disjoint curves} \label{sec:disjoint}

In the previous section we carefully analyzed what happens when $f\colon S \to T$ contracts a single rational curve of self-intersection $(-n)$. The goal of this section is to understand what happens when $f$ contracts many disjoint curves.

We consider the following setup. Let $S$ be a smooth, projective surface, and let $f\colon S \to T$ be a birational morphism $f\colon S \to T$, where $T$ is a normal, projective surface. We assume that $\Exc(f) = C_1 \cup \dots \cup C_r$, where each $C_i$ is a smooth, rational curve with self-intersection $-n_i$. We assume that the curves $C_i$ are pairwise disjoint, and that $n_i \geq 3$.

\begin{remark}
Given integers $n_1, \dots, n_r$ with $n_i \geq 3$, there always exist a surface $S$ satisfying the previous assumption. For example, take a normal, projective surface $\overline{T}$ containing a single singular point: a cyclic quotient singularity associated to the Hirzebruch--Jung continued fraction $[n_1, 2, n_2, 2, \dots, n_{r-1}, 2, n_r]$. Consider its minimal resolution $S \to \overline{T}$, and pick every other curve in the exceptional divisor. This gives us curves $C_1, \dots, C_r$ satisfying the requirements.
\end{remark}

Fix $\eta \in \Amp(T)_\Q$ and $\beta \in \NS(S)_\Q$ satisfying\footnote{For example, take $\beta=\sum_{i=1}^r (1/2 + 1/(2n_i))C_i$.} $-1 <\beta.C_i - n_i/2 <0$. We denote by $\overline{\sigma}_{\beta, f^\ast\eta}$ the stability condition obtained from Theorem \ref{teo:limits_all}. Lastly, denote by $v=[pt]$ the class of a point.

\begin{prop} \label{prop:disjoint_main}
Under the previous assumptions, there are $r$ walls $W_1, \dots, W_r$ for $v$ passing though $\overline{\sigma}_{\beta, f^\ast\eta}$. These are transversal, hence there are $2^r$ chambers around $\overline{\sigma}_{\beta, f^\ast\eta}$. Label each chamber $\Cchamb_I$ by a subset $I \subset \{1, \dots, r\}$, where $i \in I$ if $\Cchamb_I$ and the geometric chamber lie on different sides of $W_i$.

Lastly, let $\sigma$ be a stability condition on the chamber $\Cchamb_I$. The moduli space of $\sigma$-stable objects with numerical class $v$ is isomorphic to $S \cup \bigcup_{i \in I} \P^{n_i-1}$, where $C_i \subset S$ is glued with a rational normal curve in $\P^{n_i-1}$, and there are no other identifications. 
\end{prop}

The proof of Proposition \ref{prop:disjoint_main} will take the remainder of this section. We will start by computing the walls for $v$ passing through $\overline{\sigma}_{\beta, f^\ast\eta}$. The same argument of Subsection \ref{subsec:singless} shows that the $\overline{\sigma}_{\beta, f^\ast\eta}$-polystable objects are 
\[ \{ \O_x: x \in S \setminus (C_1 \cup \dots \cup C_r) \} \cup \{ \O_{C_i} \oplus \O_{C_i}(-1)[1] : 1 \leq i \leq r \}. \]
Each of the objects $\O_{C_i} \oplus \O_{C_i}(-1)[1]$ defines a numerical wall $W_i$, which can be computed as the preimage of 
\begin{equation} \label{eq:disjoint_wall}
\{Z \in \Hom(\Lambda, \C) : \Re Z(\O_{C_i}) \cdot \Im Z([pt]) = \Re Z([pt]) \cdot \Im Z(\O_{C_i}) \}
\end{equation}
under the projection $\Stab(S) \to \Hom(\Lambda, \C)$. This way, to show that the walls $W_i$ are transversal, we can compute it directly in $\Hom(\Lambda, \C)$. This is a straightforward computation.

\begin{remark} \label{remark:disjoint_slice}
Given $\epsilon_1, \dots, \epsilon_r$ small real numbers, we let $\omega = f^\ast\eta + \epsilon_1 C_1 + \dots + \epsilon_r C_r$. Let $\sigma_\epsilon$ be the deformation of $\overline{\sigma}_{\beta, f^\ast\eta}$ with central charge $Z_{\beta, \omega}$. This defines an $r$-dimensional manifold inside $\Stab(S)$, containing $\overline{\sigma}_{\beta, f^\ast\eta}$ in its interior.

Note that for $\epsilon_1, \dots, \epsilon_r <0$ small enough, we have that $\omega$ is ample. Thus, we have that $\sigma_\epsilon$ agrees with the Arcara--Bertram stability condition $\sigma_{\beta, \omega}$. This immediately describes the geometric chamber. 

With this in mind, we can compute the intersection of $W_i$ with this submanifold, by using \eqref{eq:disjoint_wall}. It turns out that they are described by the equations $\epsilon_i =0$. This gives an alternative way to verify the transversality.
\end{remark}

From the transversality of the walls $W_i$ we get that they determine $2^r$ regions of $\Stab(S)$ around $\overline{\sigma}_{\beta, f^\ast\eta}$. This way, to each chamber $\Cchamb$ we assign the subset $I \subset \{1, \dots, r\}$, where $i \in I$ if $\Cchamb_I$ and the geometric chamber lie in different sides of the wall $W_i$. We point out that the geometric chamber is assigned the label $\varnothing$.

\begin{remark}
In the notation of Remark \ref{remark:disjoint_slice}, we have that $\sigma_\epsilon \in \Cchamb_I$ if $\epsilon_i \neq 0$ for all $i$, and $I = \{1 \leq i \leq r: \epsilon_i >0\}$.
\end{remark}

To finish the proof, we need to compute the moduli space $M_I = M_{\sigma}([\O_{pt}])$, where $\sigma \in \Cchamb_I$. The case $I=\varnothing$ is part of Proposition \ref{prop:AB_moduli}, while for $\# I = 1$ it is a direct consequence of the discussion in Section \ref{sec:single}. 

Now, recall that the objects parametrized by $\P^{n_i-1}$ do not admit maps from $\O_{C_j}(-1)[1]$ for $j \neq i$, as they have disjoint support. In other words, when we cross the wall $W_j$, only objects over $C_j \subset S$ (and $\P^{n_j-1}$) can be destabilized. This way, the modifications performed by $W_j$ do not affect the irreducible component $\P^{n_i-1}$, nor the surface $S$ in a Zariski open neighborhood of $C_i$. This immediately proves the claim about $M_I = S \cup \bigcup_{i \in I} \P^{n_i-1}$, as the description of Section \ref{sec:single} applies locally. The same argument gives the description at the walls.

\section{Two intersecting curves} \label{sec:inter}

The goal of this section is to study what happens when $f\colon S \to T$ contracts a chain of two rational curves to a single point. We let $S$ be a smooth, projective surface, $f\colon S \to T$ a birational morphism to a normal, projective surface. We assume that $\Exc(f) = C_1 \cup C_2$, where each $C_i$ is a smooth, rational curve of self-intersection $-n_i \leq -3$. We also assume that $C_1$ and $C_2$ intersect at a single point. 

Pick $\eta \in \Amp(T)_\Q$, and $\beta \in \NS(S)_\Q$ satisfying the conditions
\[ -1 < \beta.C_i -n_i/2 <0, \quad \beta.(C_1+C_2) - (n_1+n_2)/2 <-1. \]
We let $\overline{\sigma}_{\beta, f^\ast\eta}$ be the stability condition given by Theorem \ref{teo:limits_all}. Set $v=[pt]$.

\begin{prop} \label{prop:inter_main}
There are three walls $W_1, W_2, W_{12}$ for $v$ passing through $\overline{\sigma}_{\beta, f^\ast\eta}$. These walls are \emph{not} transversal, and they divide $\Stab(S)$ around $\overline{\sigma}_{\beta, f^\ast\eta}$ into six chambers. Label the chambers $\Cchamb_1, \dots, \Cchamb_6$ so that $\Cchamb_1$ is the geometric chamber, that $\Cchamb_i, \Cchamb_{i+1}$ share a wall, with $W_1$ being between $\Cchamb_1$ and $\Cchamb_2$.

Denote by $M_i$ the moduli space $M_\sigma(v)$ for some $\sigma \in \Cchamb_i$. We have that $M_1=S$, $M_2 = S \sqcup_{C_1} \P^{n_1-1}$, where $C_1 \subset \P^{n_1-1}$ is embedded as a rational normal curve. We also have that $M_3 = S \cup \Bl_{pt} \P^{n_1-1} \cup \P^{n_1+n_2-3}$ has three irreducible components, glued as follows. The exceptional divisor of $\Bl_{pt} \P^{n_1-1}$ glues inside $\P^{n_1+n_2-1}$ as a linear $\P^{n_1-2}$; $C_2 \subset S$ glues as a rational normal curve in a complementary $\P^{n_2-1}$; and $C_1 \subset S$ glues along $\Bl_{pt}\P^{n_1-1}$ at the strict transform of a rational normal curve in $\P^{n_1-1}$.

Finally, the same description applies symmetrically to $\Cchamb_6$ and $\Cchamb_5$.
\end{prop}

The proof of Proposition \ref{prop:inter_main} will be divided into three steps. First, we will determine the semistable objects and walls in Subsection \ref{subsec:interwalls}. We will then use this to determine the irreducible components of the corresponding moduli spaces in Subsection \ref{subsec:interirr}. Finally, we will describe the local structure in Subsection \ref{subsec:interlocal}.

\subsection{Walls} \label{subsec:interwalls}

Let us start by computing the walls for $v$ passing through $\overline{\sigma}_{\beta, f^\ast\eta}$. From Proposition \ref{prop:limits_ssfactors}, we have that the polystable objects with numerical class $v$ are
\[ \{\O_x : x \notin C_1 \cup C_2\} \cup \{ \O_{C_{12}}(k_1, k_2) \oplus \O_{C_1}(k_1-1)[1] \oplus \O_{C_2}(k_2-1)[1] \}. \]
The object on the right define three walls, which we will denote respectively by $W_{12}, W_1, W_2$. As before, we can compute the walls in $\Hom(K^{\text{num}}(S), \C)$. We get that these walls are \emph{not} transversal, and instead define six chambers around $\overline{\sigma}_{\beta, f^\ast\eta}$. 

Let us mimic the construction in Remark \ref{remark:disjoint_slice}. Given $\epsilon_1, \epsilon_2$ small real numbers, we let $\sigma_{\epsilon_1, \epsilon_2}$ be the deformation of $\overline{\sigma}_{\beta, f^\ast\eta}$ with central charge $Z_{\beta, f^\ast\eta+\epsilon_1C_1+\epsilon_2C_2}$.

In this case, we swiftly verify that the restriction of the walls to the $(\epsilon_1, \epsilon_2)$-plane have equations $W_1 = \{ -\epsilon_1n_1+\epsilon_2 =0 \}$, $W_2 = \{ \epsilon_1 - n_2\epsilon_2 = 0 \}$, and $W_{12} = \{ \epsilon_1(n_1-1) + \epsilon_2(n_2-1)=0 \}$, thanks to the formula \eqref{eq:wc_wallformula}. It is worth pointing out that the first two equations are related to the ample cone: given $\epsilon_1, \epsilon_2$ small, we have that $f^\ast\eta + \epsilon_1 C_1 + \epsilon_2 C_2$ is ample if and only if $-\epsilon_1 n_1 + \epsilon_2 >0$ and $\epsilon_2 - n_2 \epsilon_1 >0$.

We have depicted the three walls, together with the geometric chamber, in Figure \ref{fig:interwalls_chambers}. Note that we have also labelled the chambers as in Proposition \ref{prop:inter_main}. 

\begin{figure}[htbp]
\centering
\begin{tikzpicture}[scale=0.8]
\fill[black!15!white] (-2.25,-1.125) -- (0,0) -- (-1.125,-2.25) -- (-2.25,-2.25);
\draw (1.125,2.25) -- (-1.125,-2.25) (2.25,1.125) -- (-2.25,-1.125) (2.25,-2.25) -- (-2.25,2.25);
\draw[dotted, ->] (-2.5,0) -- (2.5,0);
\draw[dotted, ->] (0,-2.5) -- (0,2.5);
\node at (2.75,0) {$\epsilon_1$};
\node at (0,2.75) {$\epsilon_2$};
\node at (-1.125,-2.6) {$W_1$};
\node at (-2.6,-1.125) {$W_2$};
\node at (-2.6,2.6) {$W_{12}$};
\node at (-1.2,-1.2) {$\Cchamb_1$};
\node at (0.25,-1.5) {$\Cchamb_2$};
\node at (1.25,-0.5) {$\Cchamb_3$};
\node at (1.2,1.2) {$\Cchamb_4$};
\node at (-0.5,1.25) {$\Cchamb_5$};
\node at (-1.5,0.25) {$\Cchamb_6$};
\end{tikzpicture}
\caption{Walls of Proposition \ref*{prop:inter_main}.}
\label{fig:interwalls_chambers}
\end{figure}

\begin{remark}
Note that the moduli spaces $M_2$ is isomorphic to $S \sqcup_{C_1} \P^{n_1-1}$, while $M_6$ is isomorphic to $S \sqcup_{C_2} \P^{n_2-1}$. This follows from the description in Section \ref{sec:single}.
\end{remark}

\subsection{Irreducible components for \texorpdfstring{$M_3$}{M3}} \label{subsec:interirr}

In this subsection we will compute the irreducible components of the moduli space $M_3$. To do so, we will use the fact that $M_2 \cong S \sqcup_{C_1} \P^{n_1-1}$. Recall here that $S \setminus C_1$ parametrizes skyscraper sheaves $\{ \O_x: x \notin C_1 \}$, while $\P^{n_1-1}$ parametrizes extensions 
\begin{equation} \label{eq:interirr_extn1}
\O_{C_1}(-1)[1] \to E \to \O_{C_1} \to \O_{C_1}(-1)[2].
\end{equation}
Recall also that the wall $W_{12}$ is described by the destabilizing subobject $\O_{C_{12}}$, the structure sheaf of the curve $C_{12}=C_1 \cup C_2$. 

We start our journey by looking at which objects of $M_2$ admit a map from $\O_{C_{12}}$. For the skyscraper sheaves parametrized by $S \setminus C_1$, this is clear: an object $\O_x$ admits a non-zero map $\O_{C_{12}} \to \O_x$ if and only if $x \in C_2 \setminus C_1$. In that case, the object $\O_x$ fits into a short exact sequence $0 \to \O_{C_{12}} \to \O_x \to \O_{C_{12}}(0, -1)[1] \to 0$.

\begin{remark}
Note that $\O_{C_{12}}(0, -1)[1]$ is an extension
\[ 0 \to \O_{C_1}(-1)[1] \to \O_{C_{12}}(0, -1)[1] \to \O_{C_2}(-1)[1] \to 0. \]
For $\sigma$ in the chamber $\Cchamb_2$, this object is $\sigma$-stable. Here $\Ext^1(\O_{C_2}(-1), \O_{C_1}(-1))$ is one-dimensional, so this extension is unique (up to isomorphism).
\end{remark}

It remains to determine which objects parametrized by $\P^{n_1-1} \subset M_2$ admit a map from $\O_{C_{12}}$. To do so, we take an extension as in \eqref{eq:interirr_extn1} and apply $\Hom(\O_{C_{12}}, -)$. We get the exact sequence
\[ 0 \to \Hom(\O_{C_{12}}, E) \to \Hom(\O_{C_{12}}, \O_{C_1}) \xrightarrow{\xi \circ -} \Ext^2(\O_{C_{12}}, \O_{C_1}(-1)). \]
Note that $\Hom(\O_{C_{12}}, \O_{C_1})$ is one-dimensional. This way, we get that $\Hom(\O_{C_{12}}, E)$ is non-zero if and only if the composition map
\[ \xi\circ -\colon \Hom(\O_{C_{12}}, \O_{C_1}) \to \Ext^2(\O_{C_{12}}, \O_{C_1}(-1)) \]
is zero. 

Now, take the short exact sequence $0 \to \O_{C_{12}} \to \O_{C_1} \to \O_{C_2}(-1)[1] \to 0$, apply both $\Hom(-, \O_{C_1})$ and $\Hom(-, \O_{C_1}(-1))$, and compare using $\xi \circ -$. We get the diagram
\begin{equation}
\begin{tikzcd}[column sep=small]
\Hom(\O_{C_1}, \O_{C_1}) \arrow[r, "\cong"] \arrow[d, hook, "\xi \circ -"] & \Hom(\O_{C_{12}}, \O_{C_1}) \arrow[d, "\xi \circ -"]
\\
\Ext^2(\O_{C_1}, \O_{C_1}(-1)) \arrow[r] & \Ext^2(\O_{C_{12}}, \O_{C_1}(-1)) \arrow[r] & 0.
\end{tikzcd}
\end{equation}
The left vertical map is zero if and only if $\xi$ is contained in the kernel of the bottom map. But this kernel is given by the image of $\Ext^1(\O_{C_2}(-1), \O_{C_1}(-1))$. This is a one-dimensional space. It follows that a single object parametrized by $\P^{n_1-1}$ admits a map from $\O_{C_{12}}$. By semi-continuity of $\dim \Hom(\O_{C_{12}}, -)$, this extension class is in the closure of $C_2 \setminus C_1$. (Alternatively, chasing the diagram above gives it explicitly.) We get the following description.

\begin{claim}
Let $E$ be an object parametrized by $M_2$. We have that $\Hom(\O_{C_{12}}, E)$ is non-zero if and only if $[E]$ lies in $C_2 \subset S \subset M_2$. 
\end{claim}

\begin{remark} \label{remark:interirr_destabilizedfromPn1}
Note that all destabilized objects fit into short exact sequences
\[ 0 \to \O_{C_{12}} \to E \to \O_{C_{12}}(0, -1)[1] \to 0. \]
In the case when $E=\O_x$ for $x \in C_2 \setminus C_1$, this is clear. If $E$ lies in $\P^{n_1-1}$, note that the extension class of the sequence above is \emph{not} an inclusion of coherent sheaves $\O_{C_{12}}(0, -1) \to \O_{C_{12}}$. Instead, it is given by the composition $\O_{C_{12}}(0, -1) \to \O_{C_2}(-1) \to \O_{C_{12}}$ where the first map is a surjection of coherent sheaves, and the second one is an inclusion of coherent sheaves. We point out that this map will correspond to (a multiple of) $(0 \oplus f_0)$ in the notation of Lemma \ref{lemma:interirr_basesOab}.
\end{remark}

Using the techniques from Subsection \ref{subsec:em}, we produce two closed embeddings $S \to M_3$ and $\Bl_{pt} \P^{n_1-1} \to M_3$, where the center of the blow-up is the point corresponding to the extension class from before. 

We can also produce a closed subscheme of $M_3$ by constructing a universal family over $\P \Ext^2(\O_{C_{12}}, \O_{C_{12}}(0, -1))$, parametrizing extensions of the form
\[ 0 \to \O_{C_{12}}(0, -1)[1] \to E \to \O_{C_{12}} \to 0. \]
We will show in a second that $\Ext^2(\O_{C_{12}}, \O_{C_{12}}(0, -1))$ is $(n_1+n_2-2)$-dimensional, hence this defines a closed subscheme $\P^{n_1+n_2-3} \to M_3$. The three subschemes $S, \Bl_{pt}\P^{n_1-1}$ and $\P^{n_1+n_2-3}$ include all closed points of $M_3$. However, it is not immediately clear how these subschemes are glued, and whether $M_3$ is reduced. 

To approach this, we mimic our discussion from Subsection \ref{subsec:singleirr}. The first step we need to adapt is to pick appropriate bases of $\Hom(\O_{C_{12}}, \O_{C_{12}}(a, b))$ for $a, b \in \Z$. To do so, we use the short exact sequences
\begin{equation} \label{eq:interirr_prebasesOab}
0 \to \O_{C_{12}}(a, b) \to \O_{C_1}(a) \oplus \O_{C_2}(b) \to \O_p \to 0,
\end{equation}
where $p = C_1 \cap C_2$. Now, fix bases $e_0, e_1$ of $\Hom(\O_{C_1}, \O_{C_1}(1))$ and $f_0, f_1$ of $\Hom(\O_{C_2}, \O_{C_2})$, satisfying $e_0|_p = f_0|_p=0$ and $e_1|_p = f_1|_p=1$. Taking the long exact sequence of $\Hom(\O_{C_{12}}, -)$ on \eqref{eq:interirr_prebasesOab} gives following description.

\begin{lemma} \label{lemma:interirr_basesOab}
Let $a, b \in \Z$.
\begin{itemize}
\item If $a, b \geq 0$, then $\Hom(\O_{C_{12}}, \O_{C_{12}}(a, b))$ has dimension $a+b+1$. A basis is given by the elements $\{ e_0^a \oplus 0, e_0^{a-1}e_1 \oplus 0, \dots, e_0 e_1^{a-1} \oplus 0, e_0^a \oplus f_1^b, 0 \oplus f_0 f_1^{b-1}, \dots, 0 \oplus f_0^b \}$.
\item If $a\geq 1$ and $b < 0$, then $\Hom(\O_{C_{12}}, \O_{C_{12}}(a, b))$ has dimension $a$. A basis is given by the elements $\{e_0^i e_1^{a-i} \oplus 0: i = a, a-1, \dots, 1\}$.
\item If $a <0$ and $b \geq 1$, then $\Hom(\O_{C_{12}}, \O_{C_{12}}(a, b))$ has dimension $b$. A basis is given by $\{ 0 \oplus f_0^j f_1^{b-j} : j=1, \dots, b \}$. 
\item Otherwise, $\Hom(\O_{C_{12}}, \O_{C_{12}}(a, b))$ is zero. 
\end{itemize}
Moreover, composition at the level of Hom-spaces corresponds to pointwise multiplication of the bases above. 
\end{lemma}

Note that in Lemma \ref{lemma:interirr_basesOab} we have picked a particular order for our bases; we will always use this order. Moreover, from Serre duality we get induced dual bases on $\Ext^2(\O_{C_{12}}(a, b), \O_{C_{12}}(a', b'))$; we will use the dual basis in the corresponding order.

\begin{lemma} \label{lemma:interirr_comprank}
Let $\xi \in \Ext^2(\O_{C_{12}}, \O_{C_{12}}(0, -1))$ be given, say
\begin{align*}
\xi ={}& a_{n_1-2, 0}(e_0^{n_1-2} \oplus 0)^\vee + \dots + a_{1, n_1-3}(e_0e_1^{n_1-3} \oplus 0)^\vee + b (e_1^{n_1-2} \oplus f_1^{n_2-1})^\vee \\
&+ c_{1, n_2-2}(0 \oplus f_0f_1^{n_2-2})^\vee + \dots + c_{n_2-1, 0} (0 \oplus f_0^{n_2-1})^\vee. 
\end{align*}
Then, the map $\xi \circ -\colon \Hom(\O_{C_{12}}(0, -1), \O_{C_{12}}) \to \Ext^2(\O_{C_{12}}(0, -1), \O_{C_{12}}(0, -1))$
has rank $0$ if and only if $\xi =0$. Moreover, it has rank $\leq 1$ if and only if one of the following two conditions happen:
\begin{itemize}
\item $a_{i, n_1-2-i}=0$ for all $1 \leq i \leq n_1-2$, $b=\lambda^{n_2-1}$ and $c_{j, n_2-1-j} = \lambda^j \mu^{n_2-1-j}$ for $j=1, \dots, n_2-1$. 
\item $c_{j, n_2-1-j}=0$ for all $1 \leq j \leq n_2-1$. 
\end{itemize}
In any other case, $\xi \circ -$ has rank 2.
\end{lemma}

\begin{proof}
We proceed as in the proof of Lemma \ref{lemma:singleirr_jumplocus}. We get that $\xi \circ -$ is given by
\[ \begin{pmatrix} a_{n_1-2, 0} & \cdots & a_{1, n_1-3} & b & c_{1, n_2-2} & \cdots & c_{n_2-2, 1} \\ 0 & \cdots & 0 & c_{1, n_2-2} & c_{2, n_2-3} & \cdots & c_{n_2-1, 0} \end{pmatrix}^t. \]
The result now follows immediately. 
\end{proof}

\begin{remark} \label{remark:interirr_identification}
We can mimic Claim \ref{claim:singleirr_Ext1} to this setup. To do so, let $\xi$ be an element such that $\xi \circ -$ has rank 1. By Lemma \ref{lemma:interirr_comprank}, we get two options.
\begin{itemize}
\item In the first case, the kernel of $\xi \circ -$ is spanned by $\phi \colon\mu (1 \oplus f_1) - \lambda (0 \oplus f_0)$. This defines a map $\phi \colon \O_{C_{12}}(0, -1) \to \O_{C_{12}}$, and so a point in $C_2\setminus C_1$ (provided that $\phi$ is injective, or equivalently, that $\mu \neq 0$). The elementary modification of the universal family of $S \subset M_2$ gives us this extension class.

\item In the second case, the kernel of $\xi \circ -$ is spanned by $(0 \oplus f_0)$. The results of Subsection \ref{subsec:em} ensure that the objects parametrized by the exceptional divisor of $\Bl_{pt} \P^{n_1-1}$ are given by these extensions.
\end{itemize}
\end{remark}

As a direct consequence, we get a map $S \cup \Bl_{pt} \P^{n_1-1} \cup \P^{n_1+n_2-3} \to M_3$, where $S \cup \Bl_{pt} \P^{n_1} \cup \P^{n_1+n_2-3}$ is glued as in Proposition \ref{prop:inter_main}. This map is bijective on closed points and on tangent vectors.

\subsection{Local structure} \label{subsec:interlocal}

The goal of this section is to show that the map $S \cup \Bl_{pt} \P^{n_1-1} \cup \P^{n_1+n_2-3} \to M_3$ constructed in the previous section is an isomorphism, by using Lemma \ref{lemma:critiso_main}. To do so, we need to show that $M_3$ is reduced. This will be done by computing $\hat{\O}_{M_3, x}$ for any $x \in M_3$.

We will focus on the point $x \in M_3$ corresponding to the triple intersection $S \cap \Bl_{pt} \P^{n_1-1} \cap \P^{n_1+n_2-3}$. To do so, we follow the same idea from Subsection \ref{subsec:singlelocal}.

We start by producing a local model from our situation. We need a (quasi-projective) surface $X$ with two curves that can be contracted to a $\frac{1}{n_1n_2-1}(1, n_2)$ cyclic quotient singularity. A toric computation gives us a local model $X$ with three charts $U_1 = \Spec \C[x, u]$, $U_2 = \Spec \C[y, v]$ and $U_3 = \Spec \C[z, w]$, glued as follows
\begin{gather*}
U_1 \cap U_2\colon \qquad \C[x, u]_x \cong \C[y, v]_y, \quad y=x^{-1}, \quad v=ux^{n_1}, \\
U_2 \cap U_3\colon \qquad \C[y, v]_v \cong \C[z, w]_w, \quad w=v^{-1}, z= yv^{n_2}, \\
U_1 \cap U_3 \colon \qquad \C[x, u]_{x, u} \cong \C[z, w]_{z, w}.
\end{gather*}
Let $C_1$ be defined by the equation $u$ in $U_1$ and $v$ in $U_2$; let $C_2$ be defined by $y$ in $U_2$, $z$ in $U_3$; and let $L$ be defined via $w$ in $U_3$. We get that $C_1, C_2$ are smooth, isomorphic to $\P^1$ and transversal; while $L$ is isomorphic to $\mathbb{A}^1$, and only cuts $C_2$ at a single point. Moreover, $C_i^2=-n_i$. Using these curves we define some line bundles, whose transition functions are summarized in Table \ref{table:interlocal_transition}.

\begin{table}[htbp]
\centering
\caption{Line bundles on $X$ and transition functions.}
\label{table:interlocal_transition}
\begin{tabular}{|c|ccc|} \hline
Line bundle & $f_{12}$ & $f_{13}$ & $f_{23}$ \\ \hline
$\O_X(-C_1)$ & $y^{n_1}$ & $u=z^{n_1}w^{n_1n_2-1}$ & $w^{-1}$ \\
$\O_X(-C_2)$ & $y^{-1}$ & $z^{-1}$ & $w^{n_2}$ \\
$\O_X(L)$ & $1$ & $w$ & $w$ \\ \hline
\end{tabular}
\end{table}

Using these line bundles, we get resolutions
\begin{equation} \label{eq:interlocal_resoncurve}
\begin{aligned}
\O_{C_{12}} &= [\O_X(-C_1-C_2) \to \O_C] \\
\O_{C_{12}}(0, -1) &= [\O_X(-C_1-C_2-L) \to \O_X(-L)].
\end{aligned}
\end{equation}
Using these resolutions we can compute $\Ext^1(\O_{C_{12}}, \O_{C_{12}}(0, -1)[1])$ via a \v{C}ech cover. It turns out that this space is $(n_1+n_2-2)$-dimensional, cf. Lemma \ref{lemma:interirr_comprank}. We can give representatives for a basis in terms of the \v{C}ech complex: take $\{ \alpha^i \}_{1 \leq i \leq n_1-2} \cup \{ \beta^j\}_{0 \leq j \leq n_2-1}$, where
\[ (\alpha^i)_{12}^{-1,-1} = y^i, \, (\alpha^i)_{13}^{-1,-1} = y^i; \qquad
(\beta^j)_{13}^{-1,-1} = y^{n_1-1}w^j, \, (\beta^j)^{-1, -1}_{23} = w^j, \]
and all other entries are zero. The images of $\alpha^i, \beta^j$ in $\Ext^1(\O_{C_{12}}, \O_{C_{12}}(0, -1)[1])$ give us a basis. We will focus on $\lambda:=\beta^0$. (In the notation of Lemma \ref{lemma:interirr_comprank}, this corresponds to $(e_1^{n_1-2} \oplus f_1^{n_2-1})^\vee$.)

The class of $\lambda$ in $\Ext^1(\O_{C_{12}}, \O_{C_{12}}(0, -1)[1])$ defines an extension
\[ \O_{C_{12}}(0, -1)[1] \to E \to \O_{C_{12}} \xrightarrow{\lambda} \O_{C_{12}}(0, -1)[2]. \]
Using the cocycle computation from above, we can construct a resolution of $E$ by locally free sheaves as follows. The entries of $\lambda$ induce an element of $\Ext^1(\O_X(-C_1-C_2), \O_X(-L))$. We get a rank two vector bundle $0 \to \O_X(-L) \to V \to \O_X(-C_1-C_2) \to 0$
that is free on $U_1, U_2, U_3$, and has transition functions
\[ e_{12} = \begin{pmatrix} 1 & 0 \\ 0 & y^{n_1-1} \end{pmatrix} \quad e_{13} = \begin{pmatrix} w^{-1} & y^{n_1-1}w^{-1} \\ 0 & uz^{-1} \end{pmatrix} \quad e_{23} = \begin{pmatrix} w^{-1} & w^{-1} \\ 0 & w^{n_2-1} \end{pmatrix}. \]
Composing with the maps from \eqref{eq:interlocal_resoncurve} we get the complex 
\begin{equation} \label{eq:interlocal_reslf}
[\O_X(-L-C_1-C_2) \to V \to \O_X].
\end{equation}
By construction, this complex is quasi-isomorphic to $E$.

Using the resolution \eqref{eq:interlocal_reslf}, we compute the semicosimplicial DGLA of $\HOM(E, E)$ with respect to $\Ucover=\{U_1, U_2, U_3\}$, which we denote by $L^\Delta$. By Lemma \ref{lemma:semiDGLA_defiso}, the deformation functor associated to $\Tot_{TW}(L^\Delta)$ computes $\Def_E$.

To compute the hull of $\Def_E$, we need to find a basis of $H^1(\Tot_{TW}(L^\Delta)) \cong \Ext^1(E, E)$ and lift it to $\Tot(L^\Delta)^1$. We mimic our argument of Lemma \ref{lemma:singlelocal_hullE}, using the exact sequence
\begin{equation} \label{eq:interlocal_Ext1}
\begin{aligned}
& 0 \to \C\lambda \to \Ext^2(\O_{C_{12}}, \O_{C_{12}}(0, -1)) \to \Ext^1(E, E) \\ &\quad \to \Hom(\O_{C_{12}}(0, -1), \O_{C_{12}}) \to \Ext^2(\O_{C_{12}}, \O_{C_{12}}) \to 0,
\end{aligned}
\end{equation}
cf. Claim \ref{claim:singleirr_Ext1} (and the discussion of Remark \ref{remark:interirr_identification}).
\begin{itemize}
\item The map $\Ext^2(\O_{C_{12}}, \O_{C_{12}}(0, -1)) \to \Ext^1(E, E)$ gives us an $(n_1+n_2-3)$-dimensional subspace of $\Ext^1(E, E)$. A \v{C}ech computation gives us representatives of a basis of this subspace as follows. First, set $\mu^i \in \Tot_{TW}(L^\Delta)^1$ with $\mu^i_0 = 0$, $\mu^i_1 = 2t_0 \, dt_0 \otimes \overline{\mu}^i$, $\mu^i_2 = 2t_0 \, dt_0 \otimes \partial_1 \overline{\mu}^i$, for
\[ (\overline{\mu}^i)_{12}^{-1,-1} = \begin{pmatrix} 0 & y^i \\ 0 & 0 \end{pmatrix}, \quad (\overline{\mu}^i)_{13}^{-1,-1} = \begin{pmatrix} 0 & y^i \\ 0 & 0 \end{pmatrix}. \]
Second, set $\eta^j \in \Tot_{TW}(L^\Delta)^1$ via the formula $\eta^j_0=0$, $\eta^j_1=2t_1 \, dt_1 \otimes \overline{\eta}^j$, and $\eta^j_2 = 2t_2 \, dt_2 \otimes \partial_0 \overline{\eta}^j$, where
\[ (\overline{\eta}^j)_{13}^{-1,-1} = \begin{pmatrix} 0 & y^{n_1-1}w^j \\ 0 & 0 \end{pmatrix}, \quad (\overline{\eta}^j)_{23}^{-1,-1} = \begin{pmatrix} 0 & w^j \\ 0 & 0 \end{pmatrix}. \]
The elements $\{ \mu^i \}_{1 \leq i \leq n_1-2}$, $\{ \eta^j \}_{1 \leq j \leq n_2-1}$ give the claimed representatives.

\item The map $\Ext^1(E, E) \to \Hom(\O_{C_{12}}(0, -1), \O_{C_{12}})$ has a one-dimensional image, which maps to zero on $\Ext^2(\O_{C_{12}}, \O_{C_{12}})$. A representative of this element is given by $\tau \in \Tot_{TW}(L^\Delta)^1$ given by $\tau_0 = 1 \otimes \overline{\tau}$, $\tau_1 = 1 \otimes \partial_0 \overline{\tau}$ and $\tau_2 = 1 \otimes \partial_0 \partial_0\overline{\tau}$, where
\begin{gather*}
\overline{\tau}_1^{-2, -1} = \begin{pmatrix} -xu \\ x^{n_1}u \end{pmatrix}, \quad \overline{\tau}_2^{-2, -1} = \begin{pmatrix} -v \\ v \end{pmatrix}, \quad \overline{\tau}_3^{-2, -1} = \begin{pmatrix} 0 \\ 1 \end{pmatrix} \\
\overline{\tau}_1^{-1, 0} = \begin{pmatrix} -x^{n_1}u & -xu \end{pmatrix}, \quad \overline{\tau}_2^{-1, 0} = \begin{pmatrix} -v & -v \end{pmatrix}, \quad \overline{\tau}_3^{-1, 0} = \begin{pmatrix} -1 & 0 \end{pmatrix}.
\end{gather*} 
\end{itemize}

This way, we have that $\{ \mu^i, \eta^j, \tau \}$ represent a basis for $H^1(\Tot_{TW}(L^\Delta))$. Let $\{p_i, q_j, r\}$ be the dual basis, so that we have the element
\[ \xi_1 = \sum_{i=1}^{n_1-2} \mu^i \otimes p_i + \sum_{j=1}^{n_2-1} \eta^j \otimes q_j + \tau \otimes r \]
in $\Tot_{TW}(L^\Delta)^1 \otimes \m_S/J_1$, where $S=\C[[p_i, q_j, r]]$ and  $J_1=\m_S^2$.

\begin{lemma} \label{lemma:interirr_order3}
We have $J_2 = \m_S^3 + (q_2r, \dots, q_{n_2-1}r)$. Moreover, there is a choice of a lift $\xi_2$ such that
\[ J_3 = \m J_2 + (p_1q_1r, \dots, p_{n_1-2}q_1r, q_2r+q_1^2r, q_3r+q_1q_2r, \dots, q_{n_2-1}r+q_1q_{n_2-2}r). \]
\end{lemma}

\begin{proof}
We compute obstructions using the description in Subsection \ref{subsec:DGLA}. Note that $[\mu^i, \mu^{i'}]$, $[\eta^j, \eta^{j'}]$, $[\tau, \tau]$, $[\mu^i, \eta^j]$ are all zero in $\Tot_{TW}(L^\Delta)^2$. We also have that $[\tau, \mu^i] = -d\theta^i$, where $\theta^i_0=t_0^2 \otimes \overline{\theta}^i$, $\theta^i_1 = t_0^2 \otimes \partial_1 \overline{\theta}^i$, $\theta^i_2 = t_0^2 \otimes \partial_1 \partial_1 \overline{\theta}^i$, and 
\[ (\overline{\theta}^i)_1^{-1, 0} = \begin{pmatrix} 0 & -ux^{n_1-i} \end{pmatrix}, \quad (\overline{\theta}^i)_1^{-2,-1} = \begin{pmatrix} -   ux^{n_1-i} \\ 0 \end{pmatrix}. \]
Similarly, we have that $[\tau, \eta^1] =- d\nu$, where $\nu_0=1\otimes \overline{\nu}$, $\nu_1 = t_1^2 \otimes (\partial_0\overline{\nu}-\partial_1 \overline{\nu}) + 1 \otimes \partial_1 \overline{\nu}$, $\nu_2 = t_2^2 \otimes (\partial_0 \partial_0 \overline{\nu} - \partial_0\partial_1 \overline{\nu}) + 1 \otimes \partial_0 \partial_1 \overline{\nu}$, with
\begin{gather*}
\overline{\nu}_1^{-2,-1} = \begin{pmatrix} 0 \\ 1 \end{pmatrix}, \quad \overline{\nu}_2^{-2,-1} = \begin{pmatrix} 0 \\ 1 \end{pmatrix}, \quad \overline{\nu}_3^{-2,-1} = \begin{pmatrix} 0 \\ w \end{pmatrix}, \\
\overline{\nu}_1^{-1, 0} = \begin{pmatrix} -1 & 0 \end{pmatrix}, \quad \overline{\nu}_2^{-1,0} = \begin{pmatrix} -1 & 0 \end{pmatrix}, \quad \overline{\nu}_3^{-1,0} = \begin{pmatrix} -w & 0 \end{pmatrix}.
\end{gather*}
On the other hand, the images of $[\tau, \eta^2], \dots, [\tau, \eta^{n_2-1}]$ are linearly independent in $H^2(\Tot(L^\Delta)) \cong \Ext^2(E, E)$. This gives us the statement about $J_2$. We use this cocycles to construct the lift
\begin{equation} \label{eq:interirr_xi2}
\xi_2 = \sum_{i=1}^{n_1-2} \mu^i \otimes p_i + \sum_{j=1}^{n_2-1} \eta^j \otimes q_j + \tau \otimes r + \sum_{i=1}^{n_1-2} \theta_i \otimes p_i r + \nu \otimes q_1r
\end{equation}
in $\Tot(L^\Delta)^1 \otimes S/J_2$.

To compute $J_3$, we lift $\xi_2$ to $\tilde{\xi}_2$ given by the same formula \eqref{eq:interirr_xi2}, now seen as an element of $\Tot(L^\Delta)^1 \otimes J_2/\m_S J_2$. We need to compute $d\tilde{\xi}_2 + \frac{1}{2}[\tilde{\xi}_2, \tilde{\xi}_2]$ and its image in $H^2(\Tot(L^\Delta)) \otimes J_2/\m_S J_2$.

One quickly verifies that $[\nu, \nu]$, $[\theta^i, \theta^{i'}]$, $[\theta^i, \mu^{i'}]$, $[\theta^i, \eta^j]$, $[\tau, \nu]$, $[\tau, \theta^i]$, $[\nu, \theta^i]$ are all zero in $\Tot_{TW}(L^\Delta)^2$. The same type of computations show that $[\nu, \eta^{n_2-1}] + d\sigma =0$ for an element $\sigma \in \Tot_{TW}(L^\Delta)^1$, and that $[\nu, \eta^j] = [\tau, \eta^{j+1}]$ for $j \leq n_2-2$. At last, we get that $\{[\nu, \mu^1], \dots, [\nu, \mu^{n_1-1}], [\tau, \eta^2], \dots, [\tau, \eta^{n_2-1}]\}$ are linearly independent in $\Ext^2(E, E)$. This way, we get that the image of $d\tilde{\xi}_2 + \frac{1}{2}[\tilde{\xi}_2, \tilde{\xi}_2]$ inside $\Ext^2(E, E) \otimes J_2/\m_S J_2$ is represented by
\[ \sum_{i=1}^{n_1-2} [\nu, \mu^i] \otimes p_i q_1 r + \sum_{j=2}^{n_2-1} [\tau, \eta^j] \otimes (q_j r+q_1q_{j-1}r) + d\sigma \otimes q_1 q_{n_2-1} r. \]
This immediately gives us the claimed description of $J_3$.
\end{proof}

\begin{cor} \label{cor:interirr_reducedugly}
If $x \in M_3$ is the point corresponding to $[E]$, then $\hat{\O}_{M_3, x}$ is reduced.
\end{cor}

\begin{proof}
Consider the morphism $\tilde{M}:=S \cup \Bl_{pt} \P^{n_1-1} \cup \P^{n_1+n_2-3} \to M_3$ from Subsection \ref{subsec:interirr}. We have $\hat{\O}_{\tilde{M}, x} \cong A:=\C[[p_i, q_j, r]]/(p_1q_1r, \dots, p_{n_1-2}q_1r, q_2r, \dots, q_{n_2-1}r)$. 

We apply Proposition \ref{prop:versal_stopcriterion} to the map $\tilde{\O}_{M_3, x} \to A$ with $d=4$. The condition $J_3=J+\m_S^4$ holds thanks to the computation of Lemma \ref{lemma:interirr_order3}.
\end{proof}

\begin{remark} \label{remark:interirr_restofpts}
A similar computation can be carried out to show that $M_3$ is reduced at all other points. It turns out that for these points one can apply Proposition \ref{prop:versal_stopcriterion} with $d=2$. Thus, one can use the techniques of \citelist{\cite{Xia18}*{\textsection 4--5} \cite{TX22}*{\textsection 8}} to compute the principal obstruction, and conclude by using Corollary \ref{cor:versal_quadratic}.
\end{remark}

\begin{cor}
The morphism $\tilde{M}:=S \cup \Bl_{pt} \P^{n_1-1} \cup \P^{n_1+n_2-3} \to M_3$ from Subsection \ref{subsec:interirr} is an isomorphism.
\end{cor}

\begin{proof}
We apply Lemma \ref{lemma:critiso_main}. We have checked before that $\tilde{M} \to M_3$ is bijective on points and on tangent vectors; the results of Corollary \ref{cor:interirr_reducedugly} and Remark \ref{remark:interirr_restofpts} ensure that $M_3$ is reduced.
\end{proof}

\bibliography{Half2.bbl}

\begin{bibdiv}
\begin{biblist}
\bib{AB13}{article}{
  author={Arcara, Daniele},
  author={Bertram, Aaron},
  title={Bridgeland-stable moduli spaces for $K$-trivial surfaces},
  note={With an appendix by Max Lieblich},
  journal={J. Eur. Math. Soc. (JEMS)},
  volume={15},
  date={2013},
  number={1},
  pages={1--38},
  issn={1435-9855},
  review={\MR {2998828}},
  doi={10.4171/JEMS/354},
}

\bib{AHLH23}{article}{
  author={Alper, Jarod},
  author={Halpern-Leistner, Daniel},
  author={Heinloth, Jochen},
  title={Existence of moduli spaces for algebraic stacks},
  journal={Invent. Math.},
  volume={234},
  date={2023},
  number={3},
  pages={949--1038},
  issn={0020-9910},
  review={\MR {4665776}},
  doi={10.1007/s00222-023-01214-4},
}

\bib{Alp25}{misc}{
  author={Alper, Jarod},
  title={Stacks and Moduli},
  eprint={https://sites.math.washington.edu/~jarod/moduli.pdf},
  year={2025-07},
}

\bib{Art76}{book}{
  author={Artin, Michael},
  title={Lectures on deformations of singularities},
  series={Lectures on mathematics and physics},
  volume={54},
  publisher={Tata Institute of Fundamental Research},
  date={1976},
  pages={ii+127},
}

\bib{AS25}{article}{
  author={Arbarello, Enrico},
  author={Sacc\`a, Giulia},
  title={Singularities of Bridgeland Moduli spaces for K3 categories: an update},
  conference={ title={Perspectives on Four Decades of Algebraic Geometry, Volume 1}, },
  book={ series={Progr. Math.}, volume={351}, publisher={Birkh\"auser/Springer, Cham}, },
  isbn={978-3-031-66229-4},
  isbn={978-3-031-66230-0},
  date={2025},
  pages={1--42},
  review={\MR {4866833}},
  doi={10.1007/978-3-031-66230-0\_1},
}

\bib{BM11}{article}{
  author={Bayer, Arend},
  author={Macr\`{i}, Emanuele},
  title={The space of stability conditions on the local projective plane},
  journal={Duke Math. J.},
  volume={160},
  date={2011},
  number={2},
  pages={263--322},
  issn={0012-7094},
  review={\MR {2852118}},
  doi={10.1215/00127094-1444249},
}

\bib{BM14}{article}{
  author={Bayer, Arend},
  author={Macr\`i, Emanuele},
  title={MMP for moduli of sheaves on K3s via wall-crossing: nef and movable cones, Lagrangian fibrations},
  journal={Invent. Math.},
  volume={198},
  date={2014},
  number={3},
  pages={505--590},
  issn={0020-9910},
  review={\MR {3279532}},
  doi={10.1007/s00222-014-0501-8},
}

\bib{BM23}{article}{
  author={Bayer, Arend},
  author={Macr\`i, Emanuele},
  title={The unreasonable effectiveness of wall-crossing in algebraic geometry},
  conference={ title={ICM---International Congress of Mathematicians. Vol. 3. Sections 1--4}, },
  book={ publisher={EMS Press, Berlin}, },
  isbn={978-3-98547-061-7},
  isbn={978-3-98547-561-2},
  isbn={978-3-98547-058-7},
  date={2023},
  pages={2172--2195},
  review={\MR {4680313}},
}

\bib{Bri07}{article}{
  author={Bridgeland, Tom},
  title={Stability conditions on triangulated categories},
  journal={Ann. of Math. (2)},
  volume={166},
  date={2007},
  number={2},
  pages={317--345},
  issn={0003-486X},
  review={\MR {2373143}},
  doi={10.4007/annals.2007.166.317},
}

\bib{Bri08}{article}{
  author={Bridgeland, Tom},
  title={Stability conditions on $K3$ surfaces},
  journal={Duke Math. J.},
  volume={141},
  date={2008},
  number={2},
  pages={241--291},
  issn={0012-7094},
  review={\MR {2376815}},
  doi={10.1215/S0012-7094-08-14122-5},
}

\bib{Cho24}{article}{
  author={Chou, Tzu-Yang},
  title={Stability condition on a singular surface and its resolution},
  year={2024},
  pages={33},
  doi={10.48550/arXiv.2411.19768},
  eprint={https://arxiv.org/abs/2411.19768},
}

\bib{CLS11}{book}{
  author={Cox, David A.},
  author={Little, John B.},
  author={Schenck, Henry K.},
  title={Toric varieties},
  series={Graduate Studies in Mathematics},
  volume={124},
  publisher={American Mathematical Society, Providence, RI},
  date={2011},
  pages={xxiv+841},
  isbn={978-0-8218-4819-7},
  review={\MR {2810322}},
  doi={10.1090/gsm/124},
}

\bib{CPZ24}{article}{
  author={Chen, Huachen},
  author={Pertusi, Laura},
  author={Zhao, Xiaolei},
  title={Some remarks about deformation theory and formality conjecture},
  journal={Ann. Univ. Ferrara Sez. VII Sci. Mat.},
  volume={70},
  date={2024},
  number={3},
  pages={761--779},
  issn={0430-3202},
  review={\MR {4767547}},
  doi={10.1007/s11565-024-00500-0},
}

\bib{Har10}{book}{
  author={Hartshorne, Robin},
  title={Deformation theory},
  series={Graduate Texts in Mathematics},
  volume={257},
  publisher={Springer, New York},
  date={2010},
  pages={viii+234},
  isbn={978-1-4419-1595-5},
  review={\MR {2583634}},
  doi={10.1007/978-1-4419-1596-2},
}

\bib{Iac10}{article}{
  author={Iacono, Donatella},
  title={Deformations of algebraic subvarieties},
  journal={Rend. Mat. Appl. (7)},
  volume={30},
  date={2010},
  number={1},
  pages={89--109},
  issn={1120-7183},
  review={\MR {2682560}},
}

\bib{Lan24}{article}{
  author={Langer, Adrian},
  title={Bridgeland stability conditions on normal surfaces},
  journal={Ann. Mat. Pura Appl. (4)},
  volume={203},
  date={2024},
  number={6},
  pages={2653--2664},
  issn={0373-3114},
  review={\MR {4813221}},
  doi={10.1007/s10231-024-01460-0},
}

\bib{Lie06}{article}{
  author={Lieblich, Max},
  title={Moduli of complexes on a proper morphism},
  journal={J. Algebraic Geom.},
  volume={15},
  date={2006},
  number={1},
  pages={175--206},
  issn={1056-3911},
  review={\MR {2177199}},
  doi={10.1090/S1056-3911-05-00418-2},
}

\bib{LS06}{article}{
  author={Lehn, Manfred},
  author={Sorger, Christoph},
  title={La singularit\'e{} de O'Grady},
  language={French, with English and French summaries},
  journal={J. Algebraic Geom.},
  volume={15},
  date={2006},
  number={4},
  pages={753--770},
  issn={1056-3911},
  review={\MR {2237269}},
  doi={10.1090/S1056-3911-06-00437-1},
}

\bib{LR22}{article}{
  author={Lim, Bronson},
  author={Rota, Franco},
  title={Characteristic classes and stability conditions for projective Kleinian orbisurfaces},
  journal={Math. Z.},
  volume={300},
  date={2022},
  number={1},
  pages={827--849},
  issn={0025-5874},
  review={\MR {4359545}},
  doi={10.1007/ s00209-021-02805-8},
}

\bib{Man09}{article}{
  author={Manetti, Marco},
  title={Differential graded Lie algebras and formal deformation theory},
  conference={ title={Algebraic geometry---Seattle 2005. Part 2}, },
  book={ series={Proc. Sympos. Pure Math.}, volume={80, Part 2}, publisher={Amer. Math. Soc., Providence, RI}, },
  isbn={978-0-8218-4703-9},
  date={2009},
  pages={785--810},
  review={\MR {2483955}},
  doi={10.1090/pspum/080.2/2483955},
}

\bib{Man22}{book}{
  author={Manetti, Marco},
  title={Lie methods in deformation theory},
  series={Springer Monographs in Mathematics},
  publisher={Springer, Singapore},
  date={2022},
  pages={xii+574},
  isbn={978-981-19-1184-2},
  isbn={978-981-19-1185-9},
  review={\MR {4485797}},
  doi={10.1007/978-981-19-1185-9},
}

\bib{PT19}{article}{
  author={Piyaratne, Dulip},
  author={Toda, Yukinobu},
  title={Moduli of Bridgeland semistable objects on 3-folds and Donaldson-Thomas invariants},
  journal={J. Reine Angew. Math.},
  volume={747},
  date={2019},
  pages={175--219},
  issn={0075-4102},
  review={\MR {3905133}},
  doi={10.1515/crelle-2016-0006},
}

\bib{Stacks}{misc}{
  author={The {Stacks Project Authors}},
  title={Stacks Project},
  eprint={https://stacks.math.columbia.edu},
  year={2025},
  label={Stacks},
}

\bib{Tod08}{article}{
  author={Toda, Yukinobu},
  title={Moduli stacks and invariants of semistable objects on $K3$ surfaces},
  journal={Adv. Math.},
  volume={217},
  date={2008},
  number={6},
  pages={2736--2781},
  issn={0001-8708},
  review={\MR {2397465}},
  doi={10.1016/j.aim.2007.11.010},
}

\bib{Tod13}{article}{
  author={Toda, Yukinobu},
  title={Stability conditions and extremal contractions},
  journal={Math. Ann.},
  volume={357},
  date={2013},
  number={2},
  pages={631--685},
  issn={0025-5831},
  review={\MR {3096520}},
  doi={10.1007/s00208-013-0915-4},
}

\bib{TX22}{article}{
  author={Tramel, Rebecca},
  author={Xia, Bingyu},
  title={Bridgeland stability conditions on surfaces with curves of negative self-intersection},
  journal={Adv. Geom.},
  volume={22},
  date={2022},
  number={3},
  pages={383--408},
  issn={1615-715X},
  review={\MR {4453506}},
  doi={10.1515/advgeom-2022-0009},
}

\bib{Vil25}{article}{
  author={Vilches, Nicolás},
  title={Stability conditions on surfaces and contractions of curves},
  year={2025},
  pages={26},
  doi={10.48550/arXiv.2508.07019},
  eprint={https://arxiv.org/abs/2508.07019},
}

\bib{Xia18}{article}{
  author={Xia, Bingyu},
  title={Hilbert scheme of twisted cubics as a simple wall-crossing},
  journal={Trans. Amer. Math. Soc.},
  volume={370},
  date={2018},
  number={8},
  pages={5535--5559},
  issn={0002-9947},
  review={\MR {3803142}},
  doi={10.1090/tran/7150},
}
\end{biblist}
\end{bibdiv}
\end{document}